\documentclass[a4]{amsart}

\input xypic 
\input xy 
\xyoption{all} 
\usepackage{amssymb} 
\usepackage{bbm}
\newtheorem{theorem}{Theorem}[section]
\newtheorem{proposition}[theorem]{Proposition}
\newtheorem{lemma}[theorem]{Lemma}
\newtheorem{corollary}[theorem]{Corollary}

\theoremstyle{definition}
\newtheorem{definition}[theorem]{Definition}
\newtheorem{remark}[theorem]{Remark}
\newtheorem{example}[theorem]{Example}

\newcommand{\conn}{\ensuremath{\#}}

\newcommand{\hlgy}[1]{\ensuremath{H_{*}(#1)}}
\newcommand{\rhlgy}[1]{\ensuremath{\widetilde{H}_{*}(#1)}}
\newcommand{\cohlgy}[1]{\ensuremath{H^{*}(#1)}}

\newcounter{bean}

\newcommand{\namedright}[3]{\ensuremath{#1\stackrel{#2}
 {\longrightarrow}#3}}
\newcommand{\nameddright}[5]{\ensuremath{#1\stackrel{#2}
 {\longrightarrow}#3\stackrel{#4}{\longrightarrow}#5}}
\newcommand{\namedddright}[7]{\ensuremath{#1\stackrel{#2}
 {\longrightarrow}#3\stackrel{#4}{\longrightarrow}#5
  \stackrel{#6}{\longrightarrow}#7}}

\newcommand{\larrow}{\relbar\!\!\relbar\!\!\rightarrow}
\newcommand{\llarrow}{\relbar\!\!\relbar\!\!\larrow}
\newcommand{\lllarrow}{\relbar\!\!\relbar\!\!\llarrow}

\newcommand{\lnamedright}[3]{\ensuremath{#1\stackrel{#2}
 {\larrow}#3}}
\newcommand{\lnameddright}[5]{\ensuremath{#1\stackrel{#2}
 {\larrow}#3\stackrel{#4}{\larrow}#5}}
\newcommand{\lnamedddright}[7]{\ensuremath{#1\stackrel{#2}
 {\larrow}#3\stackrel{#4}{\larrow}#5
  \stackrel{#6}{\larrow}#7}}

\newcommand{\llnameddright}[5]{\ensuremath{#1\stackrel{#2}
 {\llarrow}#3\stackrel{#4}{\llarrow}#5}}

\newcommand{\lllnamedright}[3]{\ensuremath{#1\stackrel{#2}
 {\lllarrow}#3}}
\newcommand{\lllnameddright}[5]{\ensuremath{#1\stackrel{#2}
 {\lllarrow}#3\stackrel{#4}{\lllarrow}#5}}

\newcommand{\qqed}{\hfill\Box}

\begin{document}


\title[Homotopy groups of Poincar\'{e} Duality complexes] 
   {Homotopy groups of highly connected Poincar\'{e} Duality complexes} 

\author{Piotr Beben} 
\address{International Technegroup Ltd, Swavesey, Cambridge CB24 4UQ, United Kingdom} 
\email{pdbcas@gmail.com} 
\author{Stephen Theriault}
\address{Mathematical Sciences, University of Southampton, Southampton 
   SO17 1BJ, United Kingdom}
\email{S.D.Theriault@soton.ac.uk}



\begin{abstract} 
Methods are developed to relate the action of a principal fibration 
to relative Whitehead products in order to determine the homotopy type 
of certain spaces. The methods are applied to thoroughly analyze the homotopy 
type of the based loops on certain cell attachments. Key examples are 
$(n-1)$-connected Poincar\'{e} Duality complexes of dimension $2n$ or $2n+1$  
with minor cohomological conditions.  
\end{abstract}

\maketitle

\section{Introduction} 

A long-standing problem in homotopy theory is to determine the 
effect on homotopy type from attaching a cell. Starting with a space $X$ 
one considers a cofibration 
\(\nameddright{S^{m}}{f}{X}{}{X\cup CS^{m}}\) 
where, for a space $A$, $CA$ is the reduced cone on $A$. 
The aim is to determine the homotopy type of $X\cup CS^{m}$, or 
$\Omega(X\cup CS^{m})$, from information on the homotopy type 
of $X$ and the homotopy class of $f$. Rational homotopy 
theory has developed methods for dealing with certain kinds of cell 
attachments, called inert or lazy~\cite{FT, HaL, HeL}. These methods do not translate 
well to the $p$-local case, except for primes that are large relative to 
the dimension of $X$ divided by its connectivity~(see, for example,~\cite{A}), 
and therefore give limited information on the integral homotopy type. 

If cell attachments are generalized to attaching a cone, and the context is 
specialized somewhat, then information can sometimes be obtained. For 
example, suppose that there is a fibration 
\(\nameddright{F}{}{E}{}{B}\). 
Then the map 
\(\namedright{E}{}{B}\) 
extends to a map 
\(\namedright{E\cup CF}{}{B}\). 
If $G$ is the homotopy fibre of this new map, then Ganea~\cite{Ga} showed that 
$G\simeq F\ast\Omega B$, where the right side is the join of $F$ and $\Omega B$. 
Further, he showed that there is a homotopy equivalence 
$\Omega(E\cup CF)\simeq\Omega B\times\Omega (F\ast\Omega B)$. 
Note that this holds integrally. Ganea's result can be recovered as a special 
case of Mather's Cube Lemma~\cite{M}, and the Cube Lemma has been 
used in a wide variety of contexts to identify the integral homotopy types of 
certain spaces. 

We consider the following case related to the cell attachment problem and 
Ganea-type results. Suppose that there is a cofibration 
\(\nameddright{A}{f}{Y}{}{Y\cup CA}\); 
there is a map 
\(\namedright{Y}{}{Z}\) 
which induces a principal fibration 
\(\nameddright{\Omega Z}{}{E}{p}{Y}\); 
and the map 
\(\namedright{Y}{}{Z}\) 
extends to a map  
\(\namedright{Y\cup CA}{}{Z}\), 
inducing a principal fibration 
\(\nameddright{\Omega Z}{}{E'}{p'}{Y\cup CA}\). 
In Sections~\ref{sec:principal} and~\ref{sec:relcube} 
we develop new techniques that relate the action of a principal fibration 
to relative Whitehead products in order to identify the homotopy type 
of $E'$ and the homotopy class of $p'$ in terms of the homotopy type 
of $E$ and the homotopy classes of~$p$ and $f$. This requires certain 
hypotheses on the spaces and maps involved, but these are fulfilled in 
a wide variety of contexts. 

The new methods are powerful and should have numerous applications. 
We use them to prove general results about certain cell attachments in 
Theorems~\ref{PDthm} and~\ref{PDthm2}. Philosophically, these cell attachments  
are integral versions of inert maps in rational homotopy theory. Key examples 
are $(n-1)$-connected $2n$-dimensional Poincar\'{e} Duality complexes $M$ having 
$H^{n}(M;\mathbb{Z}))\cong\mathbb{Z}^{d}$ for $d\geq 2$ and $n\notin\{4,8\}$, 
and $(n-1)$-connected 
$(2n+1)$-dimensional Poincar\'{e} Duality complexes $M$ having 
$H^{n+1}(M;\mathbb{Z}))\cong\mathbb{Z}^{d}$ for $d\geq 1$. 
In Examples~\ref{decompex1} and~\ref{decompex2} 
we give an explicit homotopy decomposition of $\Omega M$ and identify 
the homotopy classes of the maps from the factors into $\Omega M$. 
In Examples~\ref{decompex11} and~\ref{decompex22} we go further: 
if $\overline{M}$ is $M$ with a point removed and 
\(\namedright{\overline{M}}{j}{M}\) 
is the inclusion, then we explicitly identify the homotopy fibre of $j$, the 
homotopy class of the map from the fibre into $\overline{M}$, and show that 
$\Omega j$ has a right homotopy inverse. Collectively, these results give a 
thorough picture of the homotopy theory associated to~$M$. They subsume 
most of the results in~\cite{B,BB,BT,BW} and often go much further. For example, 
the statements about the map $j$ were known only in the case when $M$ 
is a simply-connected $4$-manifold. Finally, in Section~\ref{sec:collar} we 
give an entirely new example that gives conditions for when the map 
\(\namedright{M\conn N}{}{M\vee N}\)  
that collapses the collar in a connected sum of two Poincar\'{e} duality complexes 
has a right homotopy inverse after looping. 

The authors would like to thank the referee for a careful reading of the paper 
and numerous valuable comments.

\section{Principal fibrations and relative Whitehead products} 
\label{sec:principal} 

In this section we define relative Whitehead products and relate 
them to the action induced by a principal fibration. This will be 
an important technical tool used later in the paper. The material 
in this section is a development of that in~\cite[\S 3]{Gr}, which in 
turn was a development on~\cite[\S 6.10]{N}. We give a full account in order 
to have to hand all the material needed for later. 

In what follows, it should be pointed out that by a fibration we 
mean a strict fibration, not a fibration up to homotopy. 
All spaces are assumed to be path-connected and pointed, with the basepoint 
generically denoted by~$\ast$, and to have the homotopy type of a finite type $CW$-complex. 
For a space $X$, let $X^{I}$ be the space 
of (not necessarily pointed) continuous maps from the unit interval $I$ 
to~$X$. Let $PX=\{\omega\in X^{I}\mid \omega(0)=\ast\}$ be the 
\emph{path space} of $X$. Let 
\(ev_{1}\colon\namedright{PX}{}{X}\) 
be the evaluation map, defined by $ev_{1}(\omega)=\omega(1)$. 
The \emph{loop space} $\Omega X$ is the subspace of $PX$ 
with the property that $ev_{1}(\omega)=\ast$. It is well known that 
there is a fibration 
\(\nameddright{\Omega X}{}{PX}{ev_{1}}{X}\).  

Let 
\(\nameddright{\Omega Z}{i}{E}{p}{B}\) 
be a principal fibration induced by a map 
\(\varphi\colon\namedright{B}{}{Z}\). 
Precisely, $E$ and $p$ are defined by the pullback 
\[\diagram 
       E\rto\dto^{p} & PZ\dto^{ev_{1}} \\ 
       B\rto^-{\varphi} & Z.  
  \enddiagram\]
So $E=\{(b,\omega)\in B\times PZ\mid \omega(1)=b\}$  
and $p$ is the projection, $p(b,\omega)=b$. 

This principal fibration has an action of the fibre on the total space, 
\[a\colon\namedright{\Omega Z\times E}{}{E}\] 
defined by $a(\gamma,(b,\omega))=(b,\omega\circ\gamma)$. 
One useful property this satisfies is the following. The definition 
of $p$ as a projection implies that there is a strictly commutative 
diagram 
\begin{equation} 
  \label{actionproj} 
  \diagram 
       \Omega Z\times E\rto^-{a}\dto^{\pi_{2}} & E\dto^{p} \\ 
       E\rto^-{p} & B.  
  \enddiagram 
\end{equation} 

Consider the map 
\(\namedright{B\vee E}{1\vee p}{B}\). 
Since the composite 
\(\nameddright{E}{p}{B}{\varphi}{Z}\) 
is null homotopic, there is a homotopy commutative diagram 
\[\diagram 
       B\vee E\rto^-{1\vee p}\dto^{p_{1}} & B\dto^{\varphi} \\ 
       B\rto^-{\varphi} & Z 
  \enddiagram\] 
where $p_{1}$ is the pinch map onto the first wedge summand. 
We wish to have a model for the homotopy fibre of $p_{1}$. 

In general, if 
\(f\colon\namedright{X}{}{Y}\) 
is a continuous map with $Y$ connected, there is a standard 
way of turning $f$ into a fibration. Define the space $\widetilde{X}$ 
as the pullback 
\[\diagram 
        \widetilde{X}\rto\dto & Y^{I}\dto^{ev_{1}} \\ 
        X\rto^-{f} & Y 
  \enddiagram\] 
where $ev_{1}(\omega)=\omega(1)$. Then 
$\widetilde{X}=\{(x,\omega)\in X\times Y^{I}\mid f(x)=\omega(1)\}$. 
Let $\widetilde{f}$ be the composite 
\[\widetilde{f}\colon\nameddright{\widetilde{X}}{}{Y^{I}}{ev_{0}}{Y}\] 
where $ev_{0}(\omega)=\omega(0)$. Then $\widetilde{f}$ is a fibration. 
Define 
\(\iota\colon\namedright{X}{}{\widetilde{X}}\) 
by $\iota(x)=(x,c_{f(x)})$ where $c_{f(x)}$ is the constant path 
at $f(x)$. Then there is a strictly commutative diagram 
\[\diagram 
         X\rto^-{\iota}\dto_{f} & \widetilde{X}\dlto^-{\widetilde{f}} \\ 
         Y & 
  \enddiagram\] 
in which $\widetilde{f}$ is a fibration and $\iota$ is a homotopy 
equivalence. The \emph{homotopy fibre} of $f$ is the actual 
fibre of $\widetilde{f}$, which is the space 
$F_{f}=\{(x,\omega)\in\widetilde{X}\mid\omega(0)=\ast\}$, and this 
maps to $X$ by the composite 
\(\nameddright{F_{f}}{}{\widetilde{X}}{}{X}\). 
Note that if this construction is applied to the inclusion of the basepoint 
\(\namedright{\ast}{}{Y}\) 
then $\widetilde{\ast}$ is the ``reverse" path space on $Y$, 
$\widetilde{\ast}=\{\omega\in Y^{I}\mid \omega(1)=\ast\}$, 
and $F_{f}=\Omega Y$. If the construction is applied to the identity map 
\(\namedright{Y}{}{Y}\) 
then $F_{f}$ is precisely the path space $PY$. 

Apply this construction to the pinch map 
\(\namedright{B\vee E}{p_{1}}{B}\).  
The restriction of $p_{1}$ to $B$ is the identity map on $B$. 
So the part of the fibre $F_{p_{1}}$ corresponding to $B\subseteq B\vee E$ 
is $PB$ and this maps to $B$ by sending $\gamma\in PB$ to $\gamma(1)\in B$. 
The restriction of $p_{1}$ to $E$ is the constant map to the basepoint, 
so the part of the fibre $F_{p_{1}}$ corresponding to $E\subseteq B\vee E$ is 
$\Omega B\times E$, and this maps to $E$ by projecting $(\gamma,e)$ 
to $e$. The two parts of the fibre $F_{p_{1}}$ that correspond to the basepoint 
$\ast\subseteq B\vee E$ match at~$\Omega B$. Thus 
$F_{p_{1}}=PB\cup_{\Omega B}\Omega B\times E$ and the map 
\(\namedright{PB\cup_{\Omega B}\Omega B\times E}{}{B\vee E}\) 
is given by sending $\gamma\in PB$ to $\gamma(1)\in B$ 
and projecting $(\gamma,e)\in\Omega B\times E$ to $e\in E$. The 
initial model for the homotopy fibre of~$p_{1}$ is therefore 
$PB\cup_{\Omega B}\Omega B\times E$. 

It is convenient to express this homotopy fibre in terms 
of the cone on $\Omega B$, up to homotopy equivalence. For a 
space $Y$, the \emph{reduced cone} on $Y$ is defined by 
$CY=(Y\times I)/\sim$ where $(y,0)\sim\ast$ and $(\ast,t)\sim\ast$. 
Observe that $Y$ includes into $CY$ by sending $y$ to $(y,1)$. 
Notice that $CY$ is a lower cone, which we use instead of the more 
usual upper cone, as it makes several subsequent formulas easier 
to follow. It is well known that the map of pairs 
\[\xi\colon\namedright{(C\Omega B,\Omega B)}{}{(PB,\Omega B)}\] 
defined by $\xi(\gamma,t)(s)=\gamma(st)$ is a homotopy equivalence. 

Collecting all the information above, and noting that all the constructions 
involved are natural, we obtain the following. 

\begin{lemma} 
   \label{fibmodel} 
   A model for the homotopy fibre of the pinch map 
   \(\namedright{B\vee E}{p_{1}}{B}\) 
   is $C\Omega B\cup_{\Omega B} \Omega B\times E$, and with respect to 
   this model the map 
   \[\psi\colon\namedright{C\Omega B\cup_{\Omega B}\Omega B\times E}{}{B\vee E}\]  
   from the fibre is given by sending $(\gamma,t)\in C\Omega B$ to 
   $\gamma(t)\in B$ and projecting $(\gamma,e)\in\Omega B\times E$ 
   to~$e\in E$. Further, all of this is natural for maps 
   \(\namedright{B\vee E}{}{B'\vee E'}\).~$\qqed$ 
\end{lemma} 
       
From this description we can immediately determine the following composition. 

\begin{corollary} 
   \label{Gammalower} 
   The composite 
   \(\nameddright{C\Omega B\cup_{\Omega B}\Omega B\times E}{\psi}{B\vee E} 
          {1\vee p}{B}\) 
   is given by sending $(\gamma,t)\in C\Omega B$ to $\gamma(t)\in B$ 
   and by sending $(\gamma,e)\in\Omega B\times E$ to $p(e)$.~$\qqed$ 
\end{corollary} 

Next, we wish to produce an alternative description of the composition 
in Corollary~\ref{Gammalower} which depends on the action of the 
principal fibration 
\(\nameddright{\Omega Z}{i}{E}{p}{B}\). 
The composite 
\(\nameddright{\Omega B}{\Omega\varphi}{\Omega Z}{i}{E}\)  
is null homotopic since it is the composition of two consecutive 
maps in a homotopy fibration sequence. An explicit null homotopy 
is as follows. Recall that 
$E=\{(b,\omega)\in B\times PZ\mid \omega(1)=\varphi(b)\}$. 
Define 
\[H\colon\namedright{\Omega B\times I}{}{E}\] 
by $H(\gamma,t)=(\gamma(t),(\Omega\varphi)(\gamma_{t}))$ 
where $\gamma_{t}(s)=\gamma(st)$. Notice that $\gamma_{0}$ 
is the constant map and $\gamma_{1}=\gamma$. Observe that:  
\begin{itemize} 
   \item[(i)] $(\Omega\varphi)(\gamma_{t})(1)=\varphi(\gamma_{t}(1))=\varphi(\gamma(t))$; 
   \item[(ii)] $H(\gamma,0)=(\gamma(0),(\Omega\varphi)(\gamma_{0})=(\ast,\ast)$; 
   \item[(iii)] $H(\gamma,1)=(\gamma(1),(\Omega\varphi)(\gamma_{1})= 
             (\ast,(\Omega\varphi)(\gamma))=(i\circ\Omega\varphi)(\gamma)$. 
\end{itemize} 
Item (i) implies that $H(\gamma,t)\in E$ so $H$ is well-defined, item (ii) 
implies that $H_{0}$ is the constant map and item (iii) implies that 
$H_{1}=i\circ\Omega\varphi$. 

Recalling that $C\Omega B$ is a lower cone with $\Omega B$ including in 
by sending $b$ to $(b,1)$, the homotopy~$H$ can be used to define a map 
\(K\colon\namedright{C\Omega B}{}{E}\) 
by $K(\gamma,t)=H(\gamma,t)$. Then there is a strictly commutative diagram 
\[\diagram 
       \Omega B\rto^-{\Omega\varphi}\dto & \Omega Z\rto^-{i} & E \\ 
       C\Omega B\urrto_-{K}. 
  \enddiagram\] 
Consider the composite 
\(\nameddright{C\Omega B}{K}{E}{p}{B}\). 
As $E=\{(b,\omega)\in B\times PZ\mid \omega(1)=\varphi(b)\}$ and 
$p$ is the projection $p(b,\omega)=b$, we obtain  
\begin{equation} 
  \label{pev} 
  (p\circ K)(\gamma,t)=(p\circ H)(\gamma,t)= 
       p((\gamma(t),(\Omega\varphi)(\gamma_{t})))=\gamma(t). 
\end{equation}  
That is, $p\circ K$ is the evaluation map. 

We relate $K$ to the action 
\(\namedright{\Omega Z\times E}{a}{E}\) 
for the principal fibration 
\(\nameddright{\Omega Z}{i}{E}{p}{B}\). 
Let~$\theta$ be the composite 
\[\theta\colon\lnameddright{\Omega B\times E}{\Omega\varphi\times 1} 
     {\Omega Z\times E}{a}{E}.\] 
Since the restriction of $a$ to $\Omega Z$ is $i$, the restriction of $\theta$ 
to $\Omega B$ is $i\circ\Omega\varphi$. On the other hand, by definition 
of $K$ in terms of the homotopy $H$ and item~(iii) above, the restriction of $K$ to 
$\Omega B\subseteq C\Omega B$ is $i\circ\Omega\varphi$. Therefore 
there is a pushout map 
\begin{equation} 
   \label{Gammadgrm} 
   \xymatrix{  
         \Omega B\ar[r]\ar[d] & \Omega B\times E\ar[d]\ar@/^/[ddrr]^{\theta} & & \\ 
         C\Omega B\ar[r]\ar@/_/[drrr]_{K} & C\Omega B\cup_{\Omega B}\Omega B\times E 
             \ar@{-->}[drr]^{\Gamma} & & \\
         & & & E.} 
\end{equation} 
that defines $\Gamma$. 

\begin{lemma} 
   \label{Gammaupper} 
   The composite 
   \(\nameddright{C\Omega B\cup_{\Omega B}\Omega B\times E}{\Gamma} 
       {E}{p}{B}\) 
   is given by sending $(\gamma,t)\in C\Omega B$ to $\gamma(t)\in B$ 
   and by sending $(\gamma,e)\in\Omega B\times E$ to $p(e)$.  
\end{lemma} 

\begin{proof} 
The restriction of $\Gamma$ to $C\Omega B$ is $K$, so the restriction of 
$p\circ\Gamma$ to $C\Omega B$ is $p\circ K$, which by~(\ref{pev}) is the evaluation map 
sending $(\gamma,t)$ to $\gamma(t)$. The restriction of $\Gamma$ to 
$\Omega B\times E$ is $\theta$, so the restriction of $p\circ\Gamma$ 
to $\Omega B\times E$ is $p\circ\theta=p\circ a\circ(\Omega\varphi\times 1)$. 
Using~(\ref{actionproj}), there is a strictly commutative diagram 
\[\diagram 
       \Omega B\times E\rto^-{\Omega\varphi\times 1}\dto^{\pi_{2}} 
            & \Omega Z\times E\rto^-{a}\dto^{\pi_{2}} & E\dto^{p} \\ 
       E\rdouble & E\rto^-{p} & B  
  \enddiagram\] 
which shows that $p\circ\theta$ sends $(\gamma,e)\in\Omega B\times E$ 
to $p(e)$. 
\end{proof} 

Corollary~\ref{Gammalower} and Lemma~\ref{Gammaupper} combine to give the following. 

\begin{proposition} 
   \label{keysquare} 
   There is a strictly commutative diagram 
   \[\diagram 
         C\Omega B\cup_{\Omega B}\Omega B\times E\rto^-{\Gamma}\dto^{\psi}  
             & E\dto^{p} \\ 
         B\vee E\rto^-{1\vee p} & B. 
     \enddiagram\] 
   $\qqed$ 
\end{proposition} 

Proposition~\ref{keysquare} is a key technical result. It relates the homotopy 
fibre of the pinch map from $B\vee E$ to $B$ to the action induced 
by the principal fibration $p$. Its importance will be seen in how 
it is used to relate certain Whitehead products on $B$ to the 
principal action. 

Let $G$ be an $H$-group, which is a homotopy associative $H$-space 
with a homotopy inverse. For example, any loop space is an $H$-group. 
By~\cite{S}, the multiplication $m$ on $G$ is homotopic to one 
in which the basepoint is a strict unit, that is, one for which $m(x,\ast)=x=m(\ast,x)$ 
for any $x\in G$ (this is true for any $H$-space, not just $H$-groups). From here on it will 
be assumed that the multiplication on $G$ has a strict unit. If there are maps 
\(f\colon\namedright{A}{}{G}\) 
and 
\(g\colon\namedright{B}{}{G}\) 
the commutator 
\(c(f,g)\colon\namedright{A\times B}{}{G}\) 
of $f$ and $g$ is defined pointwise by 
$c(f,g)(a,b)=f(a)g(b)f(a)^{-1}g(b)^{-1}$. Since the multiplication 
on $G$ has a strict unit, the restriction of $c(f,g)$ to $G\vee G$ is null 
homotopic, so $c(f,g)$ extends to a map 
\[\langle f,g\rangle\colon\namedright{A\wedge B}{}{G}\] 
called the \emph{Samelson product} of $f$ and $g$. As the connecting map 
for the homotopy cofibration 
\(\nameddright{A\vee B}{}{A\times B}{}{A\wedge B}\) 
is null homotopic (since $\Sigma A\vee\Sigma B$ retracts off $\Sigma(A\times B)$), 
the homotopy class of $\langle f,g\rangle$ is determined uniquely by that of $c(f,g)$. 

In general, if $X$ and $Y$ are path-connected spaces then there are homeomorphisms 
$\Sigma(X\wedge Y)\cong\Sigma X\wedge Y\cong X\wedge\Sigma Y$. In what follows 
the version used will depend only on aesthetics. Carrying on from the previous paragraph, 
if $G=\Omega Z$ then $f,g$ have adjoints 
\(f'\colon\namedright{\Sigma A}{}{Z}\) 
and 
\(g'\colon\namedright{\Sigma B}{}{Z}\) 
respectively. The \emph{Whitehead product} of $f'$ and $g'$ is 
the map 
\([f',g']\colon\namedright{\Sigma A\wedge B}{}{Z}\) 
obtained by taking the adjoint of the Samelson product $\langle f,g\rangle$. 

Let $X$ and $Y$ be path-connected spaces. Let 
\(i_{1}\colon\namedright{\Sigma X}{}{\Sigma X\vee\Sigma Y}\) 
and 
\(i_{2}\colon\namedright{\Sigma Y}{}{\Sigma X\vee\Sigma Y}\) 
be the inclusion of the left and right wedge summands respectively. Let 
\(W\colon\namedright{\Sigma X\wedge Y}{}{\Sigma X\vee\Sigma Y}\) 
be the Whitehead product $W=[i_{1},i_{2}]$. It is well known that there 
is a homotopy cofibration 
\(\nameddright{\Sigma X\wedge Y}{W}{\Sigma X\vee\Sigma Y}{} 
     {\Sigma X\times\Sigma Y}\) 
where the right map is the inclusion of the wedge into the product. 
Observe that the pinch map 
\(p_{1}\colon\namedright{\Sigma X\vee\Sigma Y}{}{\Sigma X}\) 
factors as the composite 
\(\nameddright{\Sigma X\vee\Sigma Y}{}{\Sigma X\times\Sigma Y}{\pi_{1}} 
     {\Sigma X}\) 
where $\pi_{1}$ is the projection onto the first factor. Thus $p_{1}\circ W$ 
is null homotopic, implying that $W$ lifts to the homotopy fibre of $p_{1}$. 
Using the model for the homotopy fibre of $p_{1}$ already established, we 
obtain a lift 
\begin{equation} 
  \label{lambdalift} 
  \diagram 
         & C\Omega\Sigma X\cup_{\Omega\Sigma X}\Omega\Sigma X\times\Sigma Y\dto^{\psi} \\ 
         \Sigma X\wedge Y\rto^-{W}\urto^-{\lambda} & \Sigma X\vee\Sigma Y  
  \enddiagram 
\end{equation} 
for some map $\lambda$. In the homotopy fibration sequence 
\[\nameddright{\Omega(\Sigma X\vee\Sigma Y)}{\Omega p_{1}}{\Omega\Sigma X} 
      {\partial}{C\Omega\Sigma X\cup_{\Omega\Sigma X}\Omega\Sigma X\times\Sigma Y},\]  
where $\partial$ is the fibration connecting map, the map $\Omega p_{1}$ 
has a right homotopy inverse, implying that~$\partial$ is null homotopic. 
Thus the homotopy class of the lift $\lambda$ is uniquely determined by 
the homotopy class of $W$. The naturality of $W$ therefore implies the 
naturality of the homotopy class of the lift $\lambda$.  

We develop this in the context of the wedge $B\vee E$ used 
previously. Suppose that there are maps 
\(f\colon\namedright{\Sigma X}{}{B}\) 
and 
\(g\colon\namedright{\Sigma Y}{}{E}\). 
Consider the diagram 
\begin{equation} 
  \label{reldef} 
  \diagram 
        & C\Omega\Sigma X\cup_{\Omega\Sigma X}\Omega\Sigma X\times\Sigma Y 
                   \rrto^-{\Theta}\dto^{\psi} 
           & & C\Omega B\cup_{\Omega B}\Omega B\times E\rto^-{\Gamma}\dto^{\psi}  
           & E\dto^{p} \\ 
        \Sigma X\wedge Y\rto^-{W}\urto^-{\lambda} 
           & \Sigma X\vee\Sigma Y\rrto^-{f\vee g} 
           & & B\vee E\rto^-{1\vee p} & B  
  \enddiagram 
\end{equation} 
where $\Theta=C\Omega f\cup_{\Omega f}(\Omega f\times g)$. 
The left triangle commutes by~(\ref{lambdalift}). The middle square 
strictly commutes by the naturality of Lemma~\ref{fibmodel}. The right square 
strictly commutes by Proposition~\ref{keysquare}. Observe that the 
naturality of the Whitehead product $W=[i_{1},i_{2}]$ implies that 
the composite along the bottom row of~(\ref{reldef}) is homotopic 
to the Whitehead product $[f,p\circ g]$. Thus $\Gamma\circ\Theta\circ\lambda$ 
is a lift of $[f,p\circ g]$ through $p$. 

\begin{definition} 
The composite $[f,g]_{r}=\Gamma\circ\Theta\circ\lambda$ is the 
\emph{relative Whitehead product} of the maps $f$ and $g$. 
\end{definition} 

\begin{remark} 
The relative Whitehead product generalizes the usual Whitehead product. If the 
underlying homotopy fibration 
\(\nameddright{E}{p}{B}{}{Z}\) 
has $Z=\ast$ then $E=B$ and $p$ is the identity map, so~(\ref{reldef}) implies that 
$[f,g]_{r}$ is homotopic to the composite along the lower row of the diagram, which 
with $p=1$ is the usual Whitehead product $[f,g]$. 
\end{remark} 

\begin{remark} 
\label{natremark1} 
The naturality of the construction of $\Gamma$ and $\Theta$, and the 
naturality of the homotopy class of $\lambda$, implies that the relative 
Whitehead product is natural, up to homotopy, for maps of principal fibrations 
\[\diagram 
        E\rto\dto^{p} & E'\dto^{p'} \\
        B\rto & B' 
  \enddiagram\] 
and maps 
\[\diagram 
        \Sigma X\rto^-{f}\dto^{\Sigma a} & B\dto 
           & & \Sigma Y\rto^-{g}\dto^{\Sigma b} & B\dto \\ 
        \Sigma X'\rto^-{f'} & B' & & \Sigma Y'\rto^-{g'} & B'. 
  \enddiagram\] 
\end{remark} 

It will be useful in what follows to now introduce some homotopies. 
Let $X$ and $Y$ be path-connected pointed spaces. Let $X\ltimes Y$ be 
the \emph{left half-smash} of $X$ and $Y$, defined as the quotient space 
$(X\times Y)/\sim$ where $(x,\ast)\sim\ast$. Let 
\(q\colon\namedright{X\times Y}{}{X\ltimes Y}\) 
be the quotient map. Observe that there is a pushout diagram 
\begin{equation}
  \label{halfsmashquotient} 
  \diagram 
       X\rto\dto & X\times Y\rto^-{q}\dto & X\ltimes Y\ddouble \\ 
       CX\rto & CX\cup_{X} X\times Y\rto^-{e} & X\ltimes Y 
  \enddiagram 
\end{equation}  
where $e$ collapes the cone to a point. Since $CX$ is (naturally) contractible, 
$e$ is a natural homotopy equivalence. In our context, the space 
$C\Omega B\cup_{\Omega B}\Omega B\times E$ that is the homotopy fibre of  
the pinch map 
\(\namedright{B\vee E}{p_{1}}{B}\) 
is naturally homotopy equivalent to $\Omega B\ltimes E$. 

Let 
\begin{equation} 
  \label{epsequiv} 
  \epsilon\colon\namedright{X\ltimes Y}{}{CX\cup_{X} X\times Y} 
\end{equation}  
be a natural right homotopy inverse of $e$. Considering the spaces and 
maps in~(\ref{reldef}), let $\overline{\Gamma}$ be the composite 
\[\overline{\Gamma}\colon\nameddright{\Omega B\ltimes E}{\epsilon}  
     {C\Omega B\cup_{\Omega B}\Omega B\times E}{\Gamma}{E}\] 
and let $\overline{\lambda}$ be the composite 
\[\overline{\lambda}\colon\nameddright{\Sigma X\wedge Y}{\lambda} 
      {C\Omega\Sigma X\cup_{\Omega\Sigma X}\Omega\Sigma X\times\Sigma Y} 
      {e}{\Omega\Sigma X\ltimes\Sigma Y}.\] 
Then the definitions of $\overline{\Gamma}$, $\overline{\lambda}$ and $\Theta$, and 
the naturality of $\epsilon$, imply that there is a homotopy commutative diagram 
\begin{equation} 
  \label{htpyreldef} 
  \diagram 
      \Sigma X\wedge Y\rto^-{\overline{\lambda}}\ddouble  
          & \Omega\Sigma X\ltimes\Sigma Y\rto^-{\Omega f\ltimes g}\dto^{\epsilon} 
          & \Omega B\ltimes E\rto^-{\overline{\Gamma}}\dto^{\epsilon} & E\ddouble \\ 
      \Sigma X\wedge Y\rto^-{\lambda} 
         & C\Omega\Sigma X\cup_{\Omega\Sigma X}\Omega\Sigma X\times\Sigma Y  
                   \rto^-{\Theta} 
         & C\Omega B\cup_{\Omega B}\Omega B\times E\rto^-{\Gamma} & E. 
  \enddiagram 
\end{equation}  
Combining~(\ref{reldef}) and~(\ref{htpyreldef}) and the definition of the 
relative Whitehead product $[f,g]_{r}$ as $\Gamma\circ\Theta\circ\lambda$, 
we immediately obtain the following. 

\begin{lemma} 
   \label{reldefhtpy} 
    There is a homotopy commutative diagram  
    \[\diagram 
           & \Omega\Sigma X\ltimes\Sigma Y\rto^-{\Omega f\ltimes g}\dto 
              & \Omega B\ltimes E\rto^-{\overline{\Gamma}}\dto & E\dto^{p} \\ 
           \Sigma X\wedge Y\rto^-{W}\urto^-{\overline{\lambda}} 
              & \Sigma X\vee\Sigma Y\rto^-{f\vee g} 
              & B\vee E\rto^-{1\vee p} & B  
     \enddiagram\] 
   and the relative Whitehead product 
   \(\lnamedright{\Sigma X\wedge Y}{[f,g]_{r}}{E}\) 
   is homotopic to the composite 
   $\overline{\Gamma}\circ(\Omega f\ltimes g)\circ\overline{\lambda}$.~$\qqed$ 
\end{lemma} 

Note that 
$\overline{\Gamma}\circ(\Omega f\ltimes g)\circ\overline{\lambda}$ 
satisfies the same naturality properties as $\Gamma\circ\Theta\circ\lambda$ 
stated in Remark~\ref{natremark}.   

Consider the diagram 
\[\diagram 
       \Omega B\times E\dto\drrto^{\theta} & & \\ 
       C\Omega B\cup_{\Omega B}\Omega B\times E\rrto^(0.55){\Gamma}\dto^{e} & & E \\ 
       \Omega B\ltimes E\urrto_-{\overline{\Gamma}} & &  
  \enddiagram\]  
where $\theta$ is the composite 
\(\lnameddright{\Omega B\times E}{\Omega\varphi\times 1}  
      {\Omega Z\times E}{a}{E}\)   
and $a$ is the action from the principal fibration 
\(\nameddright{\Omega Z}{}{E}{}{B}\). 
The upper triangle commutes by~(\ref{Gammadgrm}). The lower triangle 
homotopy commutes since, by definition, $\overline{\Gamma}=\Gamma\circ\epsilon$, 
and as $\epsilon\circ e$ is homotopic to the identity map on $\Omega B\ltimes E$, 
we obtain $\overline{\Gamma}\circ e\simeq\Gamma$. Also, 
by~(\ref{halfsmashquotient}), the left vertical composite is the quotient map $q$. 
This establishes the following proposition, which encapsulates the connection 
between the action of a principal fibration and the relative Whitehead product. 

\begin{proposition} 
   \label{Gammabartheta} 
   Let 
   \(\nameddright{\Omega Z}{}{E}{}{B}\) 
   be a principal fibration induced by a map 
   \(\namedright{B}{\varphi}{Z}\). 
   Let~$\theta$ be the composite 
   \(\lnameddright{\Omega B\times E}{\Omega\varphi\times 1}{\Omega Z\times E}{a}{E}\) 
   where $a$ is the action associated to the principal fibration. 
   Then there is a homotopy commutative diagram 
   \[\diagram 
          \Omega B\times E\dto^{q}\drrto^{\theta} & & \\ 
          \Omega B\ltimes E\rrto^(0.55){\overline{\Gamma}} & & E  
     \enddiagram\]  
   which is natural for maps of principal fibrations. 
\end{proposition}  

Next, we aim to better identify the maps involved in the homotopy 
commutative diagram in Lemma~\ref{reldefhtpy}. This requires two 
preliminary lemmas. 

\begin{lemma} 
   \label{halfsmash} 
   Let $Q$ and $R$ be pointed spaces. Then there is a homotopy 
   equivalence $Q\ltimes\Sigma R\simeq (\Sigma Q\wedge R)\vee\Sigma R$. 
   Further, this decomposition is natural for maps 
   \(\namedright{Q}{}{Q'}\) 
   and 
   \(\namedright{R}{}{R'}\). 
\end{lemma} 

\begin{proof} 
In general, if $T$ is a pointed space then there is a homeomorphism 
$Q\ltimes T\cong (Q_{+})\wedge T$, where $Q_{+}$ is the union of $Q$ and a disjoint 
basepoint. Note that $\Sigma(Q_{+})\simeq\Sigma Q\vee S^{1}$. Therefore, if $T=\Sigma R$ 
then there are homotopy equivalences 
\[Q\ltimes\Sigma R\cong (Q_{+})\wedge\Sigma R\simeq\Sigma(Q_{+})\wedge R\simeq 
     (\Sigma Q\vee S^{1})\wedge R\simeq (\Sigma Q\wedge R)\vee\Sigma R.\] 
The naturality of the decomposition follows from the naturality of each homotopy equivalence.   
\end{proof} 

In general, if $X$ and $Y$ are path-connected spaces and 
\(j_{1}\colon\namedright{X}{}{X\vee Y}\) 
and 
\(j_{2}\colon\namedright{Y}{}{X\vee Y}\) 
are the inclusions, then Ganea~\cite{Ga} showed that there is 
a homotopy fibration 
\[\llnameddright{\Sigma\Omega X\wedge\Omega Y}{[ev_{1},ev_{2}]}{X\vee Y}{}{X\times Y}\] 
where the right map is the inclusion of the wedge into the product and $[ev_{1},ev_{2}]$ 
is the Whitehead product of the composites 
\(ev_{1}\colon\nameddright{\Sigma\Omega X}{ev}{X}{j_{1}}{X\vee Y}\) 
and 
\(ev_{2}\colon
\nameddright{\Sigma\Omega Y}{ev}{Y}{i_{2}}{X\vee Y}\). 
Further, this fibration splits after looping, that is, $\Omega [ev_{1},ev_{2}]$ has a left 
homotopy inverse. We will use Ganea's result to prove the following lemma.     
      
\begin{lemma} 
   \label{halfsmashsusp} 
   Let $X$ and $Y$ be path-connected spaces. Then there is a homotopy equivalence 
   \(e\colon\namedright{(\Sigma\Omega X\wedge Y)\vee\Sigma Y}{}{\Omega X\ltimes\Sigma Y}\) 
   that satisfies a homotopy commutative diagram 
   \[\diagram 
          (\Sigma\Omega X\wedge Y)\vee \Sigma Y 
                      \rrto^-{e}\drto^(0.6){[ev_{1},j_{2}]+j_{2}} 
               & & \Omega X\ltimes\Sigma Y\dlto \\ 
          & X\vee\Sigma Y. & 
      \enddiagram\]  
      Further, this homotopy equivalence is natural for maps 
      \(\namedright{X}{}{X'}\) 
      and 
      \(\namedright{Y}{}{Y'}\).  
\end{lemma}  
        
\begin{proof} 
Since the pinch map 
\(\namedright{X\vee\Sigma Y}{p_{1}}{X}\) 
factors as the composite 
\(\nameddright{X\vee\Sigma Y}{}{X\times\Sigma Y}{\pi_{1}}{X}\) 
where $\pi_{1}$ is the projection, using Ganea's result we obtain a homotopy fibration diagram 
\begin{equation} 
  \label{abevcompare} 
  \diagram 
       \Sigma\Omega X\wedge\Omega\Sigma Y\rto^-{a}\ddouble & \Omega X\ltimes\Sigma Y\rto^-{b}\dto 
            & \Sigma Y\dto^{i_{2}} \\ 
       \Sigma\Omega X\wedge\Omega\Sigma Y\rto^-{[ev_{1},ev_{2}]} & X\vee\Sigma Y\rto\dto^-{p_{1}} 
            & X\times\Sigma Y\dto^{\pi_{1}} \\ 
       & X\rdouble & X  
  \enddiagram 
\end{equation} 
where $i_{2}$ is the inclusion of the second factor and $a$ and $b$ are induced by the 
fibration diagram. Let 
\(\eta\colon\namedright{Y}{}{\Omega\Sigma Y}\) 
be the suspension map, which is the adjoint of the identity map on~$\Sigma Y$, and let~$c$ 
be the composite 
\[c\colon\lnameddright{\Sigma\Omega X\wedge Y}{\Sigma 1\wedge\eta} 
     {\Sigma\Omega X\wedge\Omega\Sigma Y}{a}{\Omega X\ltimes\Sigma Y}.\] 
Let $d$ be the composite 
\[d\colon\Sigma Y\hookrightarrow\namedright{\Omega X\times\Sigma Y}{}{\Omega X\ltimes\Sigma Y}\] 
and let 
\[e\colon\namedright{(\Sigma\Omega X\wedge Y)\vee\Sigma Y}{}{\Omega X\ltimes\Sigma Y}\] 
be the wedge sum of $c$ and $d$. 

We will show that $e$ is a homotopy equivalence. 
Granting this for the moment, observe that when~$c$ is composed to $X\vee\Sigma Y$ 
the left square in~(\ref{abevcompare}) implies that we obtain $[ev_{1},ev_{2}]\circ(\Sigma 1\wedge\eta)$. 
By the naturality of the Whitehead product this is homotopic to 
$[ev_{1}, ev_{2}\circ\Sigma\eta]$. As $\Sigma\eta$ is a left homotopy inverse for 
the evaluation map and, by definition, $ev_{2}=j_{2}\circ ev$, we obtain 
$[ev_{1}, ev_{2}\circ\Sigma\eta]\simeq [ev_{1}, j_{2}]$. Next, by Lemma~\ref{fibmodel} 
and the homotopy equivalence~$\epsilon$ in~(\ref{epsequiv}), the composite 
\begin{equation} 
  \label{Eid} 
  \nameddright{\Sigma Y}{d}{\Omega X\ltimes\Sigma Y}{}{X\vee\Sigma Y} 
\end{equation} 
is the inclusion $j_{2}$ of the right wedge summand. Therefore, as $e$ is the wedge 
sum of $c$ and $d$, the composite 
\(\nameddright{(\Sigma\Omega X\wedge Y)\vee\Sigma Y}{e}{\Omega X\ltimes\Sigma Y} 
     {}{X\vee\Sigma Y}\) 
is homotopic to $[ev_{1},j_{2}]+j_{2}$, giving the homotopy commutative diagram 
asserted by the lemma. Naturality then follows by the naturality of the homotopy 
fibration diagram~(\ref{abevcompare}) and of the maps $c$ and $d$. 
 
It remains to show that $e$ is a homotopy equivalence. To do so it suffices to show 
that $\Omega e$ is a homotopy equivalence, for then $\Omega e$ induces an isomorphism 
of homotopy groups and hence so does~$e$, implying that $e$ is a homotopy equivalence. 
(Technically, this is a weak homotopy equivalence, but as all spaces have the homotopy type of 
finite type $CW$-complexes, we obtain an honest homotopy equivalence.) To show that 
$\Omega e$ is a homotopy equivalence, by Dror's extension of Whitehead's Theorem to 
path-connected $H$-spaces~\cite{D}, it suffices to show that 
$\Omega e$ induces an isomorphism in integral homology. Finally, to see this it suffices to show 
that $\Omega e$ induces an isomorphism in homology with field coefficients for any field. 
Assume from now on that homology is taken with field coefficients. 

First, we show that $\Omega c$ has a left homotopy inverse. The homotopy commutativity 
of the upper left square in~(\ref{abevcompare}) implies it suffices to show that the composite 
\[\lllnameddright{\Omega(\Sigma\Omega X\wedge Y)}{\Omega(\Sigma 1\wedge\eta)} 
     {\Omega(\Sigma\Omega X\wedge\Omega\Sigma Y)}{\Omega [ev_{1},ev_{2}]}{\Omega (X\vee\Sigma Y)}\] 
has a left homotopy inverse. Since $\Omega [ev_{1},ev_{2}]$ has a left homotopy inverse, it 
therefore suffices to show that $\Omega(\Sigma 1\wedge\eta)$ has a left homotopy inverse, 
and to show this it suffices to show that~$\Sigma 1\wedge\eta$ has a left homotopy inverse. 
But this is the case since the evaluation map 
\(\namedright{\Sigma\Omega\Sigma Y}{ev}{\Sigma Y}\) 
is a left homotopy inverse for $\Sigma\eta$ and therefore $1\wedge ev$ is a left homotopy 
inverse for ~$1\wedge\Sigma\eta\simeq\Sigma 1\wedge\eta$. 

Second, we show that $d$ is a left homotopy inverse for $b$. As noted above, the composite 
\(\nameddright{\Sigma Y}{d}{\Omega X\ltimes\Sigma Y}{}{X\vee\Sigma Y}\) 
is the inclusion of the right wedge summand, so composing further with  
\(\nameddright{X\vee\Sigma Y}{}{X\times\Sigma Y}{\pi_{2}}{\Sigma Y}\) 
gives the identity map on $\Sigma Y$. The homotopy commutativity of 
the upper right square in~(\ref{abevcompare}) then implies that  
\(\nameddright{\Sigma Y}{d}{\Omega X\ltimes\Sigma Y}{b}{\Sigma Y}\) 
is homotopic to the identity map. 

Now consider homology. In general, if $B$ is a path-connected space and homology is taken 
with a coefficient ring that is a PID, then the Bott-Samelson 
Theorem states that there is an algebra isomorphism $\hlgy{\Omega\Sigma B}\cong T(\rhlgy{B})$, 
where $T(\ \ )$ is the free tensor algebra functor, and the suspension map 
\(\namedright{B}{\eta}{\Omega\Sigma B}\) 
induces the inclusion of the generating set in homology. In our case, Lemma~\ref{halfsmash} 
implies that $\Omega X\ltimes\Sigma Y$ is homotopy equivalent to the suspension 
$\Sigma((\Omega X\wedge Y)\vee Y)$ and the coefficient ring in homology is a field, 
so there is an algebra isomorphism 
$\hlgy{\Omega(\Omega X\ltimes\Sigma Y)}\cong T(V)$, where $V=\rhlgy{(\Omega X\wedge Y)\vee Y}$.  
Since $\Omega c$ has a left homotopy inverse, $(\Omega c)_{\ast}$ is a monomorphism. 
Consequently, as the adjoint of $c$ is homotopic to the composite 
\[\tilde{c}\colon\nameddright{\Omega X\wedge Y}{\eta}{\Omega\Sigma(\Omega X\wedge Y)}
      {\Omega c}{\Omega(\Omega X\ltimes\Sigma Y)}\] 
and $\eta_{\ast}$ is also a monomorphism, we obtain that $\tilde{c}_{\ast}$ is a monomorphism. 
Since $\Omega d$ has a left homotopy inverse, $(\Omega d)_{\ast}$ is a monomorphism. 
Consequently, as the adjoint of $d$ is homotopic to the composite 
\[\tilde{d}\colon\nameddright{Y}{\eta}{\Omega\Sigma Y}{\Omega d}{\Omega(\Omega X\ltimes\Sigma Y)}.\] 
and $\eta_{\ast}$ is also a monomorphism, we obtain that $\tilde{d}_{\ast}$ is a monomorphism. 
Further, since $d$ is the inclusion of $\Sigma Y$ 
into $\Omega X\ltimes\Sigma Y$, the image of $\tilde{d}_{\ast}$ is the submodule $\rhlgy{Y}\subseteq V$. 
Since $b$ is a left homotopy inverse for $d$, the composite 
\(\nameddright{Y}{\tilde{d}}{\Omega(\Omega X\ltimes\Sigma Y)}{\Omega b}{\Omega\Sigma Y}\) 
is a monomorphism with image $\rhlgy{Y}$. In~(\ref{abevcompare}) the composite $b\circ a$ is 
null homotopic since these are consecutive maps in a homotopy fibration, so as $c$ factors 
through $a$, we have $\Omega b\circ\tilde{c}$ null homotopic, and hence 
$(\Omega b)_{\ast}\circ\tilde{c}_{\ast}=0$. Therefore, in reduced homology, the images $C$ 
and $D$ of $\tilde{c}_{\ast}$ and $\tilde{d}$ are disjoint. Thus, as both $\tilde{c}_{\ast}$ and 
$\tilde{d}_{\ast}$ are monomorphisms in reduced homology, their wedge sum is as well and  
it has image $C\oplus D\subseteq T(V)$. As the tensor algebra is free, this implies that there 
is a monomorphism 
\(\namedright{T(C\oplus D)}{}{T(V)}\). 
But this monomorphism is realized by the multiplicative extension (via the James construction) of 
\[\namedright{(\Omega X\wedge Y)\vee Y}{\tilde{c}\vee\tilde{d}}{\Omega(\Omega X\ltimes\Sigma Y)}\] 
to 
\[\upsilon\colon\namedright{\Omega\Sigma((\Omega X\wedge Y)\vee Y)}{} 
     {\Omega(\Omega X\ltimes\Sigma Y)}.\] 
Since the domain and range of $\upsilon$ have the same homotopy type, the fact 
that $\upsilon_{\ast}$ is a monomorphism implies that it is in fact an isomorphism. 
Finally, observe that the definitions of $\tilde{c}$ and $\tilde{d}$ as adjoints implies 
that $\upsilon\simeq\Omega(c\vee d)$, that is, $\upsilon\simeq\Omega e$. Hence 
$\Omega e$ induces an isomorphism in homology. 
\end{proof} 
      
Returning to Lemma~\ref{reldefhtpy} and incorporating Lemma~\ref{halfsmashsusp} 
we obtain a homotopy commutative diagram 
\[\diagram 
      (\Sigma\Omega\Sigma X\wedge Y)\vee\Sigma Y 
             \rto^-{e}\drto_{[ev_{1},j_{2}]+j_{2}} 
        & \Omega\Sigma X\ltimes\Sigma Y\rto^-{\Omega f\ltimes g}\dto 
        & \Omega B\ltimes E\rto^-{\overline{\Gamma}}\dto & E\dto^{p} \\ 
        & \Sigma X\vee\Sigma Y\rto^-{f\vee g} 
        & B\vee E\rto^-{1\vee p} & B.  
\enddiagram\] 
The naturality of the Whitehead product implies that the composite 
in the lower direction of this diagram is homotopic to 
$[f\circ ev,p\circ g]+p\circ g$. We record this as follows. 

\begin{corollary} 
   \label{Whid} 
   There is a homotopy commutative diagram 
    \[\diagram 
        \Omega\Sigma X\ltimes\Sigma Y\rto^-{\Omega f\ltimes g}\dto^{e^{-1}}  
           & \Omega B\ltimes E\rto^-{\overline{\Gamma}} & E\dto^{p} \\ 
        (\Sigma\Omega\Sigma X\wedge Y)\vee\Sigma Y 
              \rrto^-{[f\circ ev,p\circ g]+p\circ g} 
           & & B  
  \enddiagram\] 
\end{corollary}  
\vspace{-0.8cm}$\qqed$\medskip 

Corollary~\ref{Whid} says that the map $\overline{\Gamma}\circ(\Omega f\ltimes g)$ 
appearing in the definition of a relative Whitehead product is itself already a lift 
of Whitehead products, up to a homotopy equivalence. 

\begin{remark} 
\label{natremark} 
The naturality at each stage of the construction implies that the homotopy 
commutative diagram in Corollary~\ref{Whid} satisfies the same naturality properties  
listed in Remark~\ref{natremark1}. 
\end{remark} 
         
Finally, we better identify the map $\overline{\lambda}$ in the other part of 
the diagram in Lemma~\ref{reldefhtpy}. Let 
\(\eta\colon\namedright{X}{}{\Omega\Sigma X}\) 
be the suspension map, which is the adjoint of the identity map on $\Sigma X$. 

\begin{proposition} 
   \label{lambdaid} 
   There is a homotopy commutative diagram 
   \[\diagram 
            & \Sigma\Omega\Sigma X\wedge Y\rto^-{include} 
               & (\Omega\Sigma X\wedge\Sigma Y)\vee\Sigma Y\rto^-{e} 
               & \Omega\Sigma X\ltimes\Sigma Y\dto \\ 
            \Sigma X\wedge Y\xto[rrr]^-{W}\urto^{\Sigma\eta\wedge 1} 
               & & & \Sigma X\vee\Sigma Y. 
     \enddiagram\] 
   Consequently, the map 
   \(\namedright{\Sigma X\wedge Y}{\overline{\lambda}}{\Omega\Sigma X\ltimes\Sigma Y}\)  
   in Lemma~\ref{reldefhtpy} can be chosen to be the composite 
   \(L\colon\lnamedright{\Sigma X\wedge Y}{\Sigma\eta\wedge 1} 
         {\Sigma\Omega\Sigma X\wedge\Sigma Y}\hookrightarrow 
         \namedright{(\Omega\Sigma X\wedge\Sigma Y)\vee\Sigma Y}{e} 
         {\Omega\Sigma X\ltimes Y}\). 
\end{proposition} 

\begin{proof} 
Recall that $W$ is the Whitehead product $[j_{1},j_{2}]$. Let $L$ be the composite 
comprising the upper direction around the diagram in the statement of the proposition. 
By Lemma~\ref{halfsmashsusp}, $L\simeq [ev_{1},j_{2}]\circ(\Sigma\eta\wedge 1)$. 
By the naturality of the Whitehead 
product, this is homotopic to $[ev_{1}\circ\Sigma\eta, j_{2}]$. As the composite 
\(\nameddright{\Sigma X}{\Sigma\eta}{\Sigma\Omega\Sigma X}{ev}{\Sigma X}\) 
is homotopic to the identity map, we obtain  
$ev_{1}\circ\Sigma\eta=j_{1}\circ ev\circ\Sigma\eta\simeq j_{1}$. Thus 
$L\simeq [j_{1},j_{2}]=W$. 

Consequently, $L$ is a lift of $W$ to the homotopy fibre of the pinch map 
\(\namedright{\Sigma X\vee\Sigma Y}{}{\Sigma Y}\). 
As the loop of the pinch map has a right homotopy inverse, the homotopy 
class of the lift of $W$ to the homotopy fibre is uniquely determined by 
the homotopy class of $W$. Thus we may unambiguously choose the 
lift $\overline{\lambda}$ of $W$ in Lemma~(\ref{htpyreldef}) to be $L$. 
\end{proof}

\section{Relative Whitehead products and the homotopy type of certain pushouts} 
\label{sec:relcube} 

This section fuses a construction in~\cite{GT} with relative Whitehead products. 
Suppose that there is a cofibration 
\(\nameddright{A}{f}{Y}{}{Y\cup CA}\).  
Let 
\(\nameddright{\Omega Z}{}{E'}{}{Y\cup CA}\) 
be a principal fibration induced by a map 
\(\varphi\colon\namedright{Y\cup CA}{}{Z}\).  
Define spaces $Q$ and $E$ and maps~$p$ and $e$ by the iterated pullback 
\begin{equation} 
  \label{DLiterated} 
  \diagram 
      Q\rto\dto & E\rto^-{e}\dto^{p} & E'\dto \\ 
      A\rto^-{f} & Y\rto & Y\cup CA.  
  \enddiagram 
\end{equation}  
Observe that there are principal fibrations 
\(\nameddright{\Omega Z}{}{Q}{}{A}\) 
and 
\(\nameddright{\Omega Z}{}{E}{}{Y}\) 
induced by the composites 
\(\namedddright{A}{f}{Y}{}{Y\cup CA}{\varphi}{Z}\) 
and 
\(\nameddright{Y}{}{Y\cup CA}{\varphi}{Z}\) 
respectively. The map 
\(\namedright{A}{}{Z}\) 
is trivial since it factors through two consecutive maps 
in a cofibration. Thus $Q\simeq\Omega Z\times A$. 
However, there may be many inequivalent choices of a decomposition 
and we wish to choose one such that $E'$ is the pushout of the projection 
\(\namedright{\Omega Z\times A}{}{A}\) 
and an ``action" map 
\(\namedright{\Omega Z\times A}{}{E}\). 

\begin{lemma} 
   \label{Alift} 
   There is a map of pairs 
   \(\namedright{(CA,A)}{}{(E',E)}\) 
   such that the composite 
   \(\nameddright{(CA,A)}{}{(E',E)}{}{(Y\cup CA,Y)}\) 
   is the standard inclusion. 
\end{lemma} 

\begin{proof} 
Start with the standard inclusion 
\(\namedright{(CA,A)}{}{(Y\cup CA,Y)}\). 
Consider the composite\linebreak  
\(h\colon\nameddright{A\times I}{}{CA}{}{Y\cup CA}\), 
where the left map is the quotient map to $CA=A\wedge I$. The map $h$ 
is a pointed homotopy which at $t=0$ sends $A$ to the base of the cone in $Y\cup CA$ 
and at $t=1$ sends $A$ to the basepoint. Thus $h_{1}$ lifts to $E'$, so the 
homotopy extension property implies that $h$ lifts to a map 
\(\overline{h}\colon\namedright{A\times I}{}{E'}\). 
As this occurs in the pointed category, $\overline{h}$ factors as a composite 
\(\nameddright{A\times I}{}{CA}{\overline{h}'}{E'}\).  
The pullback property of $E$ then implies that there is a map 
\(g\colon\namedright{A}{}{E}\) 
such that $p\circ g$ is the identity on $A$ and $e\circ g$ is $\overline{h'}$. 
Thus $g$ and $e\circ g$ give a map of pairs 
\(\namedright{(CA,A)}{}{(E',E})\) 
with the property that the composite 
\(\nameddright{(CA,A)}{}{(E',E)}{}{(Y\cup CA,Y)}\) 
is the standard inclusion. 
\end{proof} 

Since the homotopy fibration  
\(\nameddright{\Omega Z}{}{E}{}{Y}\) 
is principal, there is a homotopy action 
\(\namedright{\Omega Z\times E}{a}{E}\). 
Let $\vartheta$ be the composite 
\[\vartheta\colon\nameddright{\Omega Z\times A}{1\times g} 
      {\Omega Z\times E}{a}{E}.\] 
The following was established in the first half of the proof 
of~\cite[Theorem 2.2]{GT}. 

\begin{theorem} 
   \label{GTpo} 
   Using the map of pairs in Lemma~\ref{Alift}, there is a homotopy pushout 
   \[\diagram 
           \Omega Z\times A\rto^-{\vartheta}\dto^{\pi_{1}} & E\dto \\ 
           \Omega Z\rto & E'. 
     \enddiagram\]  
\end{theorem} 
\vspace{-0.9cm}~$\qqed$\medskip   

The value of Theorem~\ref{GTpo} is that it allows for the homotopy 
type of $E'$ to be identified in terms of known maps. In the previous section 
we related the action $\vartheta$ to relative Whitehead products. We now 
make this explicit in the context of the pushout in Theorem~\ref{GTpo}. 

Let $h$ be the composite 
\(\nameddright{Y}{}{Y\cup CA}{\varphi}{Z}\). 
Consider the homotopy fibration sequence 
\[\namedright{\Omega Y}{\Omega h}{\Omega Z}\longrightarrow\nameddright{E}{p}  
      {Y}{h}{Z}.\]  
Suppose that the map 
\(\namedright{\Omega Z}{}{E}\) 
is null homotopic. Then $\Omega h$ has a right homotopy inverse 
\(s\colon\namedright{\Omega Z}{}{\Omega Y}\). 
Let 
\(g\colon\namedright{A}{}{E}\) 
be the restriction of $\overline{g}$ to $A$. Consider the diagram 
\[\diagram 
       \Omega Z\times A\rto^-{1\times g}\dto^{q} 
            & \Omega Z\times E\rto^-{s\times 1}\dto^{q} 
            & \Omega Y\times E\rto^-{\Omega h\times 1}\dto^{q} 
            & \Omega Z\times E\rto^-{a} & E \\ 
        \Omega Z\ltimes A\rto^-{1\ltimes g} & \Omega Z\ltimes E\rto^-{s\ltimes 1} 
             & \Omega Y\ltimes E\urrto_-{\overline{\Gamma}}  & 
  \enddiagram\] 
The left and middle squares commute by naturality and the right triangle  
homotopy commutes by Proposition~\ref{Gammabartheta}. 
Since $\Omega h\circ s$ is homotopic to the identity 
map on $\Omega Z$, the composite along the top row is homotopic 
to $a\circ(1\times g)$, which by definition, is $\vartheta$. Thus we 
obtain a homotopy commutative diagram 
\begin{equation} 
   \label{varthetafact} 
   \diagram 
         \Omega Z\times A\rrto^-{\vartheta}\dto^{q} & & E\ddouble \\ 
         \Omega Z\ltimes A\rto^-{s\ltimes g} 
              & \Omega Y\ltimes E\rto^-{\overline{\Gamma}} & E.    
     \enddiagram 
\end{equation} 
The factorization of $\vartheta$ through the half-smash $\Omega Z\ltimes A$ 
implies that, in the homotopy pushout in Theorem~\ref{GTpo}, the space $\Omega Z$ 
may be pinched out, proving the following. 

\begin{theorem} 
   \label{GTcofib} 
   Given the same hypotheses as in Theorem~\ref{GTpo}, suppose in addition 
   that the map 
   \(\namedright{\Omega Y}{\Omega h}{\Omega Z}\) 
   has a right homotopy inverse 
   \(s\colon\namedright{\Omega Z}{}{\Omega Y}\). 
   Then there is a homotopy cofibration 
   \[\lllnameddright{\Omega Z\ltimes A}{\overline{\Gamma}\circ(s\ltimes g)}{E}{}{E'}.\] 
\end{theorem} 
\vspace{-1cm}$\qqed$\bigskip 

Ideally, if the homotopy types of $\Omega Z$, $A$ and $E$ are known, and 
the homotopy class of $\overline{\Gamma}\circ(s\ltimes g)$ can 
be identified, then the homotopy type of $E'$ could also be identified. 
Since $\Omega h$ has a right homotopy inverse, the homotopy class of 
$\overline{\Gamma}\circ(s\ltimes g)$ is determined by the homotopy class 
of its composition with 
\(\namedright{E}{p}{Y}\). 
In Theorem~\ref{GTcofibWh} the homotopy class of $p\circ\overline{\Gamma}\circ(s\ltimes g)$ 
is described more precisely, provided that $A$ is a suspension. Let $\gamma$ be the 
composite 
\[\gamma\colon\nameddright{\Sigma\Omega Z}{\Sigma s}{\Sigma\Omega Y}{ev}{Y}.\] 

\begin{theorem} 
   \label{GTcofibWh} 
   Given the same hypotheses as Theorem~\ref{GTcofib}. Suppose in addition 
   that $A\simeq\Sigma\overline{A}$. Then there is a homotopy commutative diagram 
   \[\diagram 
         \Omega Z\ltimes A\rto^-{s\ltimes g}\dto^{e^{-1}} 
            & \Omega Y\ltimes E\rto^-{\overline{\Gamma}} & E\dto^{p} \\ 
         (\Sigma\Omega Z\wedge\overline{A})\vee A\rrto^-{[\gamma,f]+f} 
            & & Y   
     \enddiagram\] 
   where $e$ is the homotopy equivalence from Lemma~\ref{halfsmashsusp}. 
\end{theorem} 

\begin{proof} 
Generically, let 
\(e_{1}\colon\namedright{\Sigma\Omega Q\wedge R}{}{\Omega Q\ltimes\Sigma R}\) 
be the restriction of the homotopy equivalence 
\(\namedright{(\Sigma\Omega Q\wedge R)\vee\Sigma R}{e} 
      {\Omega Q\ltimes\Sigma R}\) 
in Lemma~\ref{halfsmashsusp} to the first wedge summand. Consider the diagram 
\[\diagram 
     \Omega Z\wedge A\rdouble\ddto^{\simeq} 
          & \Omega Z\wedge A\rto^-{e_{1}}\dto^{\eta\wedge 1} 
          & \Omega Z\ltimes A\rto^-{s\ltimes 1}\dto^{\eta\ltimes 1} 
          & \Omega Y\ltimes A\drto^{1\ltimes g}\dto^{\eta\ltimes 1} & & \\ 
      & \Omega\Sigma\Omega Z\wedge A\rto^-{e_{1}} 
          & \Omega\Sigma\Omega Z\ltimes A\rto^-{\Omega\Sigma s\ltimes 1}\dto 
          & \Omega\Sigma\Omega Y\ltimes A\rto^-{\Omega ev\ltimes g}\dto 
          & \Omega Y\ltimes E\rto^-{\overline{\Gamma}}\dto & E\dto^{p} \\ 
      \Sigma\Omega Z\wedge\overline{A}\rrto^-{W} 
          & & \Sigma\Omega Z\vee A\rto^-{\Sigma s\vee 1} 
          & \Sigma\Omega Y\vee A\rto^-{ev\vee g} 
          & Y\vee E\rto^-{1\vee p} & Y. 
  \enddiagram\] 
The $L$-shaped part of the diagram on the left homotopy commutes by 
Proposition~\ref{lambdaid}, and the square in its upper right corner homotopy commutes 
by the naturality of $e_{1}$. The upper middle square homotopy commutes 
by the naturality of $\eta$ and the upper triangle homotopy commutes since $\eta$ 
is a right homotopy inverse of $\Omega ev$. The three remaining squares on 
the bottom homotopy commute by the construction of relative Whitehead products. 
Thus the entire diagram homotopy commutes. Observe that the upper direction around 
the diagram is the restriction 
of $\overline{\Gamma}\circ(s\ltimes g)\circ e$ to $\Omega Z\wedge A$, 
while the naturality of the Whitehead product implies that the lower direction 
around the diagram is homotopic to $[ev\circ\Sigma s,p\circ g]=[\gamma,f]$. 
Thus the diagram in the statement of the theorem homotopy commutes 
when restricted to $\Sigma\Omega Z\wedge\overline{A}\simeq\Omega Z\wedge A$. 
On the other hand, the restriction of the diagram in the statement of the 
theorem to $A$ clearly homotopy commutes since the restriction of 
$\overline{\Gamma}(s\ltimes g)$ to $A$ is $g$. This completes the proof. 
\end{proof} 

A special case of Theorems~\ref{GTcofib} and~\ref{GTcofibWh} that will be 
used in the next section is given by taking  
$Z=Y\cup CA$ and  
\(\namedright{Y}{\varphi}{Z}\) 
as the identity map. This gives the following diagram of spaces and maps that 
collects the data going into Theorem~\ref{GTcofib}: 
\begin{equation} 
  \label{data} 
  \diagram 
      & E\rto\dto^{p} & E'\dto \\ 
      A\rto^-{f} & Y\rto\dto^{h} & Y\cup CA\dto^{=} \\ 
      & Y\cup CA\rdouble & Y\cup CA. 
  \enddiagram 
\end{equation} 
Here, the middle row is a cofibration, the two columns are homotopy 
fibrations, and the upper square is a homotopy pullback. The map $h$ in 
general is the composite 
\(\nameddright{Y}{}{Y\cup CA}{\varphi}{Z}\), 
so in this case it is simply the inclusion 
\(\namedright{Y}{}{Y\cup CA}\). 
Thus the space $E$ is the homotopy fibre of this inclusion and the 
space $E'$ is contractible. From Theorems~\ref{GTcofib} and~\ref{GTcofibWh} 
we obtain the following. 

\begin{proposition} 
   \label{specialGTcofib} 
   Let 
   \(\nameddright{A}{f}{Y}{h}{Y\cup CA}\) 
   be a cofibration and consider the homotopy fibration 
   \(\nameddright{E}{p}{Y}{h}{Y\cup CA}\).  
   Suppose that $\Omega h$ has a right homotopy inverse 
   \(s\colon\namedright{\Omega (Y\cup CA)}{}{\Omega Y}\).  
   Then there is a homotopy equivalence  
   \[\lllnamedright{\Omega (Y\cup CA)\ltimes A}{\overline{\Gamma}\circ(s\ltimes g)}{E}.\] 
   Further, if $A\simeq\Sigma\overline{A}$ then there is a homotopy commutative diagram 
   \[\diagram 
         \Omega (Y\cup CA)\ltimes A\rrto^-{\overline{\Gamma}\circ(s\ltimes g)}_-{\simeq} 
                   \dto^{e^{-1}}_{\simeq}  
            & & E\dto^{p} \\ 
         (\Sigma\Omega (Y\cup CA)\wedge\overline{A})\vee A\rrto^-{[\gamma,f]+f} 
            & & Y.   
     \enddiagram\] 
   Consequently, there is a homotopy fibration 
   \[\hspace{3.5cm}\llnameddright{\Sigma\Omega (Y\cup CA)\wedge\overline{A})\vee A}{[\gamma,f]+f} 
           {Y}{h}{Y\cup CA}.\hspace{3.5cm}\Box\] 
\end{proposition} 

That is, Proposition~\ref{specialGTcofib} identifies the homotopy type of $E$ 
and, if $A$ is a suspension, identifies the map 
\(\namedright{E}{p}{Y}\) 
in terms of $f$ and a Whitehead product.

\section{The based loops on highly connected Poincar\'{e} Duality complexes I} 
\label{sec:highI} 
 
In this section Proposition~\ref{specialGTcofib} is used to analyze the 
homotopy theory of certain cell attachments, as described in Theorem~\ref{PDthm}. 
This is then applied to identify the homotopy types of the based loops on 
certain Poincar\'{e} Duality complexes. By the Hilton-Milnor Theorem the inclusion 
\mbox{\(\namedright{S^{m}\vee S^{n}}{i}{S^{m}\times S^{n}}\)} 
has a right homotopy inverse after looping; that is, there is a map 
\(t\colon\namedright{\Omega S^{m}\times\Omega S^{n}}{}{\Omega(S^{m}\vee S^{n})}\) 
such that $\Omega i\circ t$ is homotopic to the identity map. 

\begin{theorem} 
   \label{PDthm} 
   Suppose that there is a cofibration 
   \(\nameddright{S^{m+n-1}}{}{Y}{}{Y'=Y\cup e^{m+n}}\) 
   where $m,n\geq 2$. Suppose also that: 
   \begin{itemize} 
      \item[(i)] there is a homotopy equivalence $Y\simeq S^{m}\vee S^{n}\vee\Sigma X$; 
      \item[(ii)] the composite of inclusions 
                    \(f\colon\Sigma X\hookrightarrow Y\hookrightarrow Y'\) 
                    has homotopy cofibre $D$ with the property that 
                    $\cohlgy{D}\cong\cohlgy{S^{m}\times S^{n}}$.   
   \end{itemize} 
   Let 
   \(q\colon\namedright{Y'}{}{D}\) 
   be the map to the cofibre and let $s$ be the composite 
   \(\nameddright{\Omega S^{m}\times\Omega S^{n}}{t} 
         {\Omega(S^{m}\vee S^{n})}{}{\Omega Y'}\). 
   Then: 
   \begin{itemize} 
      \item[(a)] the composite 
                     \(\nameddright{\Omega S^{m}\times\Omega S^{n}}{s}{\Omega Y'}{\Omega q}{\Omega D}\) 
                     is a homotopy equivalence; 
      \item[(b)] there is a homotopy fibration 
                     \[\llnameddright{(\Sigma(\Omega S^{m}\times\Omega S^{n})\wedge X)\vee\Sigma X} 
                         {[\gamma,f]+f}{Y'}{q}{D}\]  
                     where $\gamma=ev\circ\Sigma s$; 
      \item[(c)] there is a homotopy equivalence 
                     \[\Omega Y'\simeq\Omega S^{m}\times\Omega S^{n}\times 
                         \Omega((\Omega S^{m}\times\Omega S^{n})\ltimes\Sigma X).\] 
   \end{itemize} 
\end{theorem}         

\begin{proof} 
For part~(a), let $r$ be the composite 
\(r\colon S^{m}\vee S^{n}\hookrightarrow\namedright{Y'}{q}{D}\).  
The definitions of~$r$ and $s$ imply that $\Omega r\circ t=\Omega q\circ s$. 
So to prove part~(a) it is equivalent to show that $\Omega r\circ t$ is a homotopy 
equivalence. This is proved by directly copying the argument in~\cite[Lemmas 2.2 and 2.3]{BT}. 

For part~(b), consider the homotopy cofibration 
\(\nameddright{\Sigma X}{f}{Y'}{q}{D}\). 
Define the space $E$ and the map $p$ by the homotopy fibration 
\(\nameddright{E}{p}{Y'}{q}{D}\).  
Part~(a) implies the map $\Omega q$ has a right homotopy inverse 
\(s'\colon\namedright{\Omega D}{}{\Omega Y'}\). 
Therefore Proposition~\ref{specialGTcofib} implies that there is a homotopy fibration 
\[\llnameddright{(\Sigma\Omega D\wedge X)\vee\Sigma X} 
             {[\gamma',f]+f}{Y'}{q}{D}\] 
where $\gamma'=ev\circ\Sigma s'$. Substituting in the homotopy equivalence 
for $\Omega D$ in part~(a), which also lets us substutite $\gamma=ev\circ s$ 
for $\gamma'=ev\circ s'$, we obtain a homotopy fibration 
\[\llnameddright{(\Sigma(\Omega S^{m}\times\Omega S^{n})\wedge X)\vee\Sigma X} 
         {[\gamma,f]+f}{Y'}{q}{D},\] 
proving part~(b). 

Finally, for part~(c), since $\Omega q$ has a right homotopy inverse by part~(a), 
the homotopy fibration in part~(b) splits after looping, giving a homotopy equivalence 
\[\Omega Y'\simeq\Omega D\times 
            \Omega((\Omega S^{m}\times\Omega S^{n})\ltimes\Sigma X).\] 
Substituting in the homotopy equivalence 
$\Omega D\simeq\Omega S^{m}\times\Omega S^{n}$ 
from part~(a) then completes the proof. 
\end{proof} 

\begin{example} 
\label{decompex1} 
For $n\geq 2$, let $M$ be an $(n-1)$-connected $2n$-dimensional Poincar\'{e} Duality 
complex. By Poincar\'{e} Duality, 
\[H^{m}(M)\cong\left\{\begin{array}{ll} 
        \mathbb{Z} & \mbox{if $m=0$ or $m=2n$} \\ 
        \mathbb{Z}^{d} & \mbox{if $m=n$} \\ 
        0 & \mbox{otherwise}\end{array}\right.\] 
for some integer $d\geq 0$. Assume that $d\geq 2$ and $n\notin\{2,4,8\}$.  
By~\cite[Lemma 3.3]{BT} generators $x_{1},\ldots,x_{d}$ of $H^{n}(M)$ can 
be chosen such that  $x_{1}\cup x_{2}$ generates $H^{2n}(M)$ 
for some $x_{1}\neq x_{2}$. (Note that if $n\in\{2,4,8\}$ then the existence of 
an element of Hopf invariant one allows for the possibility that only $x_{1}\cup x_{1}$ 
generates $H^{2n}(M)$.) 
Now give $M$ a $CW$-structure by corresponding one $n$-cell to each $x_{k}$ 
and attaching the top cell. Then there is a homotopy cofibration 
\[\nameddright{S^{2n-1}}{g}{\bigvee_{i=1}^{d} S^{n}}{j}{M}\] 
where $j^{\ast}$ sends $x_{k}$ to the generator of the $k^{th}$ sphere in the wedge, 
and $g$ attaches the top cell of $M$. Let $\Sigma X=\bigvee_{i=3}^{d} S^{n}$ and let $f$ 
be the composite 
\(f\colon\Sigma X\hookrightarrow\namedright{\bigvee_{i=1}^{d} S^{n}}{j}{M}\). 
Define the space $D$ and the map $q$ by the homotopy cofibration 
\[\nameddright{\Sigma X}{f}{M}{q}{D}.\] 
Since $x_{1}\cup x_{2}$ generates $H^{2n}(M)$ and $x_{1},x_{2}$ correspond 
to $\bigvee_{i=1}^{2} S^{n}$, the space $D$ satisfies $\cohlgy{D}\cong\cohlgy{S^{n}\times S^{n}}$. 
Therefore, Theorem~\ref{PDthm} applies, and we obtain a homotopy fibration 
\[\llnameddright{(\Sigma(\Omega S^{n}\times\Omega S^{n})\wedge X)\vee\Sigma X} 
           {[\gamma,f]+f}{M}{q}{D}\]  
where $\gamma=ev\circ\Sigma s$ and a homotopy equivalence 
\begin{equation} 
  \label{loopM1} 
  \Omega M\simeq\Omega S^{n}\times\Omega S^{n}\times 
           \Omega((\Omega S^{n}\times\Omega S^{n})\ltimes\Sigma X). 
\end{equation} 
\end{example} 

Example~\ref{decompex1} improves on some of the relevant results in~\cite{BB,BT}. Using different 
methods, in~\cite{BT} the same decomposition for $\Omega M$ was obtained if $n\notin\{2,4,8\}$ 
and using yet another set of methods, in~\cite{BB} the same decomposition 
for $\Omega M$ was obtained for all $n$. But in neither case was the map from the 
fibre of $q$ into $M$ identified. 

\begin{remark} 
\label{sphereremark} 
In general, there are homotopy equivalences 
$B\ltimes\Sigma A\simeq (\Sigma B\wedge A)\vee\Sigma A$ and 
$\Sigma (B\times A)\simeq\Sigma B\vee\Sigma A\vee(\Sigma A\wedge B)$, and 
a property of the James construction is that $\Sigma\Omega S^{n}$ is homotopy 
equivalent to a wedge of spheres. Combining these facts shows that 
$(\Omega S^{n}\times\Omega S^{n})\ltimes\Sigma X$ is homotopy equivalent to 
a wedge of spheres, and if desired, a precise enumeration of these spheres can be 
made. The Hilton-Milnor Theorem then implies that 
$\Omega((\Omega S^{n}\times\Omega S^{n})\ltimes\Sigma X)$  
is homotopy equivalent to an infinite product of spheres. Hence the decomposition~(\ref{loopM1}) 
implies that the homotopy groups of $\Omega M$ can be determined to exactly the same 
extent as can the homotopy groups of spheres. 
\end{remark} 

\begin{example} 
\label{decompex2} 
For $n\geq 2$, let $M$ be an $(n-1)$-connected $(2n+1)$-dimensional Poincar\'{e} Duality 
complex. By Poincar\'{e} Duality, 
\[H^{m}(M)\cong\left\{\begin{array}{ll} 
        \mathbb{Z} & \mbox{if $m=0$ or $m=2n+1$} \\ 
        \mathbb{Z}^{d} & \mbox{if $m=n$} \\ 
        \mathbb{Z}^{d}\oplus G & \mbox{if $m=n+1$} \\ 
        0 & \mbox{otherwise}\end{array}\right.\] 
for some integer $d\geq 0$ and some finite abelian group $G$. Assume 
that $d\geq 1$. Rationally, $M$ still satisfies Poincar\'{e} Duality, so we can choose 
generators $x_{1},\ldots,x_{d}$ of the subgroup $\mathbb{Z}^{d}$ in $H^{n}(M)$ 
and $y_{1},\ldots,y_{d}$ of the subgroup $\mathbb{Z}^{d}$ 
in $H^{n+1}(M)$ such that  $x_{1}\cup y_{1}$ generates 
$H^{2n+1}(M)$. Give $\overline{M}$ a $CW$-structure by associating 
an $S^{n}$ to each $x_{k}$, an $S^{n+1}$ to each $y_{k}$, and an  
$(n+1)$-dimensional Moore space $P^{n+1}(t_{j})$ to each cyclic direct 
summand $\mathbb{Z}/t_{j}\mathbb{Z}$ of $G$. Write 
\begin{equation} 
  \label{2n+1Mbar} 
  \overline{M}\simeq S^{n}\vee S^{n+1}\vee\Sigma X 
\end{equation}  
where the $S^{n}$ corresponds to $x_{1}$, the $S^{n+1}$ corresponds to $y_{1}$, and  
$\Sigma X=(\bigvee_{i=2}^{d} S^{n}\vee S^{n+1})\vee(\bigvee_{j=1}^{s} P^{n+1}(t_{j}))$. 
Give $M$ a $CW$-structure by attaching the top cell to $\overline{M}$. Then  
there is a homotopy cofibration 
\[\nameddright{S^{2n}}{g}{S^{n}\vee S^{n+1}\vee\Sigma X}{j}{M}.\] 
Let $f$ be the composite 
\(f\colon\Sigma X\hookrightarrow\namedright{S^{n}\vee S^{n+1}\vee\Sigma X}{j}{M}\). 
Define the space $D$ and the map $q$ by the homotopy cofibration 
\[\nameddright{\Sigma X}{f}{M}{q}{D}.\] 
Since $x_{1}\cup y_{1}$ generates $H^{2n+1}(M)$ and $x_{1},y_{1}$ correspond 
to $S^{n}\vee S^{n+1}$ in~(\ref{2n+1Mbar}), the space $D$ satisfies 
$\cohlgy{D}\cong\cohlgy{S^{n}\times S^{n+1}}$. Therefore, Theorem~\ref{PDthm} 
applies, and we obtain a homotopy fibration 
\[\llnameddright{(\Sigma (\Omega S^{n}\times\Omega S^{n+1})\wedge X)\vee\Sigma X} 
           {[\gamma,f]+f}{M}{q}{D}\]  
where $\gamma=ev\circ\Sigma s$ and a homotopy equivalence 
\begin{equation} 
  \label{loopM2} 
  \Omega M\simeq\Omega S^{n}\times\Omega S^{n+1}\times 
           \Omega((\Omega S^{n}\times\Omega S^{n+1})\ltimes\Sigma X). 
\end{equation}  
\end{example}  

Example~\ref{decompex2} improves on the result in~\cite{B}. That paper used different 
methods to obtain the same homotopy decomposition, but did not describe the map 
from the fibre of $q$ into $M$. 

\begin{remark} 
\label{Mooreremark} 
As in Remark~\ref{sphereremark}, the fact that $\Sigma X$ is homotopy equivalent 
to a wedge of spheres and Moore spaces implies that 
$(\Omega S^{n}\times\Omega S^{n+1})\ltimes\Sigma X$ is homotopy equivalent 
to a wedge of spheres and Moore spaces. The Hilton-Milnor Theorem can then 
be applied to $\Omega((\Omega S^{n}\times\Omega S^{n+1})\ltimes\Sigma X)$ 
to decompose further. In particular, the smash product of two mod-$p^{r}$ Moore 
spaces is homotopy equivalent to a wedge of two mod-$p^{r}$ Moore spaces 
for $p$ a prime and $r\neq 2$, so if the $2$-torsion in $H^{n}(M)$ is controlled 
in this way then the output of the Hilton-Milnor Theorem is a product of looped 
spheres and looped Moore spaces. Therefore the decomposition~(\ref{loopM2}) 
implies that the homotopy groups of $M$ can be calculated to the same extent 
as can the homotopy groups of spheres and Moore spaces. 
\end{remark} 

\begin{remark} 
\label{1conn4mnfldremark} 
The $n=2$ case of simply-connected $4$-dimensional Poincar\'{e} Duality complexes 
can be recovered. In this case, we rely on the arugment in~\cite[Section 4]{BT}; this is 
phrased in terms of simply-connected $4$-manifolds but works equally well for 
simply-connected Poincar\'{e} Duality complexes. If~$M$ is such a space then 
$\Omega M\simeq S^{1}\times\Omega Z$ where $Z$ is a simply-connected 
torsion-free $5$-dimensional Poincar'{e} Duality complex. If $H^{3}(Z)=0$ then 
$Z\simeq S^{5}$ and otherwise $Z$ is one of the cases considered in 
Example~\ref{decompex2}. Therefore, in all cases, we obtain a decomposition 
of $\Omega M$. 
\end{remark}

\section{The based loops on highly connected Poincar\'{e} Duality complexes II} 
\label{sec:highII} 

Theorem~\ref{PDthm} can be pushed further. Consider again the cofibration 
\[\nameddright{S^{m+n-1}}{g}{Y}{j}{Y'=Y\cup e^{m+n}}\] 
where $g$ attaches the $(m+n)$-cell and $j$ is the inclusion. 
In Theorem~\ref{PDthm2} we identify the homotopy fibre of $j$ and show 
that $\Omega j$ has a right homotopy inverse. 

\begin{theorem} 
   \label{PDthm2} 
   Assume that there is a cofibration 
   \(\nameddright{S^{m+n-1}}{g}{Y}{j}{Y'}\) 
   as in Theorem~\ref{PDthm}. Then the following hold: 
   \begin{itemize} 
      \item[(a)] the map $\Omega j$ has a right homotopy inverse 
                     \(t\colon\namedright{\Omega M}{}{\Omega Y}\); 
      \item[(b)] there is a homotopy fibration 
                     \[\llnameddright{(\Sigma\Omega Y'\wedge S^{m+n-2})\vee S^{m+n-1}}{[\gamma,g]+g} 
                          {Y}{j}{Y'}\] 
                     where $\gamma=ev\circ\Sigma t$. 
   \end{itemize} 
\end{theorem} 

\begin{proof} 
Throughout the proof we write $Y$ as $S^{m}\vee S^{n}\vee\Sigma X$. 

For part~(a), Theorem~\ref{PDthm} gives a homotopy decomposition of $\Omega Y'$ 
via the map 
\(\namedright{Y'}{q}{D}\) 
but does not immediately relate this to $j$. To do so, let $\overline{q}$ be the composite 
\[\overline{q}\colon\nameddright{S^{m}\vee S^{n}\vee\Sigma X}{j}{Y'}{q}{D}.\] 

Recall that there is a homotopy cofibration 
\(\nameddright{\Sigma X}{f}{Y'}{q}{D}\). 
We claim that there is a homotopy cofibration 
\(\nameddright{S^{m+n-1}\vee\Sigma X}{g+i}{S^{m}\vee S^{n}\vee\Sigma X}{\overline{q}}{D}\) 
where $i$ is the inclusion of $\Sigma X$. To see this, consider the homotopy pushout diagram 
\begin{equation} 
  \label{barDdgrm} 
  \diagram 
      S^{m+n-1}\rdouble\dto^{i_{1}} & S^{m+n-1}\dto^{g} & \\ 
      S^{m+n-1}\vee\Sigma X\rto^-{g+i}\dto^{q_{2}} & S^{m}\vee S^{n}\vee\Sigma X\rto\dto^{j}  
           & \overline{D}\ddouble \\ 
      \Sigma X\rto & Y'\rto & \overline{D}.  
  \enddiagram 
\end{equation}  
Here, $i_{1}$ is the inclusion of the first wedge summand, $q_{2}$ is the pinch onto 
the second wedge summand, and $\overline{D}$ is the homotopy cofibre of $g+i$. 
The map 
\(\namedright{\Sigma X}{}{Y'}\) 
along the bottom row can be identified by restricting $j\circ(g+i)$ to $\Sigma X$: this 
is the definition of $f$. Hence $\overline{D}\simeq D$ and the homotopy cofibration 
along the bottom row of the diagram is 
\(\nameddright{\Sigma X}{f}{Y'}{q}{D}\). 
By definition, $\overline{q}=q\circ j$, so from the middle row of the diagram we 
obtain a homotopy cofibration 
\(\nameddright{S^{m+n-1}\vee\Sigma X}{g+i}{S^{m}\vee S^{n}\vee\Sigma X}{\overline{q}}{D}\). 

Next, recall from Theorem~\ref{PDthm} that the right homotopy inverse $s$ of $\Omega q$ 
is defined as the composite 
\(\nameddright{\Omega S^{m}\times\Omega S^{n}}{t}{\Omega(S^{m}\vee\Omega S^{n})} 
      {}{\Omega Y'}\).  
The latter map is the loops on the composite 
\(S^{m}\vee S^{n}\hookrightarrow\namedright{S^{m}\vee S^{n}\vee\Sigma X}{j}{Y'}\). 
Thus if $\overline{s}$ is the composite 
\[\overline{s}\colon\nameddright{\Omega S^{m}\times\Omega S^{n}}{t} 
      {\Omega(S^{m}\vee S^{n})}{}{\Omega(S^{m}\vee S^{n}\vee\Sigma X)}\] 
then $\overline{s}$ is a right homotopy inverse for $\Omega q\circ\Omega j=\Omega\overline{q}$. 
Therefore, applying Theorem~\ref{specialGTcofib} to the homotopy cofibration 
\(\nameddright{S^{m+n-1}\vee\Sigma X}{g+i}{S^{m}\vee S^{n}\vee\Sigma X}{\overline{q}}{D}\) 
we obtain a homotopy fibration 
\[\lllnameddright{(\Sigma(\Omega S^{m}\times\Omega S^{n})\wedge(S^{m+n-1}\vee\Sigma X)) 
             \vee (S^{m+n-1}\vee\Sigma X)} 
        {[\overline{\gamma},g+i]+(g+i)}{S^{m}\vee S^{n}\vee\Sigma X}{\overline{q}}{D}\] 
where $\overline{\gamma}=ev\circ\Sigma\overline{s}$. By Theorem~\ref{PDthm}, 
there is a homotopy fibration 
\[\llnameddright{(\Sigma (\Omega S^{m}\times\Omega S^{n})\wedge X)\vee\Sigma X} 
          {[\gamma,f]+f}{Y'}{q}{D}\]  
where $\gamma=ev\circ\Sigma s$. The two fibrations are compatible: (i) by definition,  
$\overline{q}=q\circ j$, (ii) by~(\ref{barDdgrm}), $j\circ(g+i)\simeq f\circ q_{2}$, 
(iii) note that $s=\Omega j\circ\overline{s}$ so $\gamma=j\circ\overline{\gamma}$, 
and (iv) the Whitehead product is natural. Thus there is a homotopy commutative diagram  
\[\diagram 
      (\Sigma(\Omega S^{m}\times\Omega S^{n})\wedge(S^{m+n-1}\vee\Sigma X))\vee (S^{m+n-1}\vee\Sigma X) 
                \rrto^-{[\overline{\gamma},g+i]+(g+i)}\dto^{(1\wedge q_{2})\vee q_{2}} 
            & & S^{m}\vee S^{n}\vee\Sigma X\rto^-{\overline{q}}\dto^{j} & D\ddouble \\ 
       (\Sigma (\Omega S^{m}\times\Omega S^{n})\wedge X)\vee\Sigma X 
               \rrto^-{[\gamma,f]+f} 
            & & Y'\rto^-{q} & D  
  \enddiagram\] 
where the two rows are homotopy fibrations. If 
\(i_{2}\colon\namedright{\Sigma X}{}{S^{m}\vee S^{n}\vee\Sigma X}\) 
is the inclusion, then $i_{2}$ is a right homotopy inverse for $q_{2}$. Therefore 
$(1\wedge i_{2})\vee i_{2}$ is a right homotopy inverse for $(1\wedge q_{2})\vee q_{2}$. 
Consequently, letting $A=(\Sigma (\Omega S^{m}\times\Omega S^{n})\wedge X)\vee\Sigma X$, 
$B=S^{m}\vee S^{n}\vee\Sigma X$, $\psi=(1\wedge i_{2})\vee i_{2}$ and $\mu$ be the loop 
multiplication, the composite 
\[\lnamedddright{\Omega A\times(\Omega S^{m}\times\Omega S^{n})}{\Omega\psi\times\overline{s}} 
      {\Omega B\times\Omega B}{\mu}{\Omega B}{\Omega j}{\Omega Y'}\] 
is a homotopy equivalence. Therefore $\Omega j$ has a right homotopy inverse. 

Part~(b) is now straightforward. The right homotopy inverse for $\Omega j$ implies 
that Theorem~\ref{PDthm} can be applied to the homotopy cofibration 
\(\nameddright{S^{m+n-1}}{g}{S^{m}\vee S^{n}\vee\Sigma X}{j}{Y'}\) 
to obtain a homotopy fibration 
\[\llnameddright{(\Sigma \Omega Y'\wedge S^{m+n-2})\vee S^{m+n-1}}{[\gamma,g]+g} 
          {S^{m}\vee S^{n}\vee\Sigma X}{j}{Y'}\] 
where $\gamma=ev\circ\Sigma t$. 
\end{proof} 

\begin{example} 
\label{decompex11} 
Let $n\geq 2$ but $n\notin\{2,4,8\}$. Let $M$ be an $(n-1)$-connected $2n$-dimensional 
Poincar\'{e} Duality complex with $H^{n}(M)\cong\mathbb{Z}^{d}$ for $d\geq 2$. As in 
Example~\ref{decompex1}, there is a homotopy cofibration 
\[\nameddright{S^{2n-1}}{g}{\bigvee_{i=1}^{d} S^{n}}{j}{M}\] 
where $g$ is the attaching map for the top cell. By Theorem~\ref{PDthm2}, 
$\Omega j$ has a right homotopy inverse 
\(t\colon\namedright{\Omega M}{}{\Omega(\bigvee_{i=1}^{d} S^{n})}\) 
and there is a homotopy fibration 
\[\llnameddright{(\Sigma\Omega M\wedge S^{2n-2})\vee S^{2n-1}}{[\gamma,g]+g} 
        {\bigvee_{i=1}^{d} S^{n}}{j}{M}\] 
where $\gamma=ev\circ\Sigma t$. 
\end{example} 

All cases in Example~\ref{decompex11} are new. Also new are all cases in the following. 

\begin{example} 
\label{decompex22} 
For $n\geq 2$, let $M$ be an $(n-1)$-connected $(2n+1)$-dimensional 
Poincar\'{e} Duality complex with $H^{n+1}(M)\cong\mathbb{Z}^{d}$ for 
$d\geq 1$. As in Example~\ref{decompex2}, there is a homotopy cofibration 
\[\nameddright{S^{2n}}{g}{S^{n}\vee S^{n+1}\vee\Sigma X}{j}{M}\] 
where $g$ is the attaching map for the top cell. By Theorem~\ref{PDthm2}, 
$\Omega j$ has a right homotopy inverse 
\(t\colon\namedright{\Omega M}{}{\Omega(S^{n}\vee S^{n+1}\vee\Sigma X)}\) 
and there is a homotopy fibration 
\[\llnameddright{(\Sigma\Omega M\wedge S^{2n-1})\vee S^{2n}}{[\gamma,g]+g} 
      {S^{n}\vee S^{n+1}\vee\Sigma X}{j}{M}\] 
where $\gamma=ev\circ\Sigma t$. 
\end{example} 

One useful application of Theorem~\ref{PDthm2} is to show that certain maps 
are null homotopic after looping. 

\begin{lemma} 
   \label{PDmap} 
   Let 
   \(\nameddright{S^{m+n-1}}{g}{Y}{j}{Y'}\) 
   be a cofibration as in Theorem~\ref{PDthm}. Suppose that there is a map 
   \(a\colon\namedright{Y'}{}{Z}\) 
   such that $a\circ j$ is null homotopic. Then $\Omega a$ is null homotopic. 
\end{lemma} 

\begin{proof} 
By Theorem~\ref{PDthm2}, $\Omega j$ has a right homotopy inverse 
\(t\colon\namedright{\Omega Y'}{}{\Omega Y}\). 
So $\Omega a\simeq\Omega a\circ\Omega j\circ t$ but $\Omega a\circ\Omega j$ 
is null homotopic, implying that $\Omega a$ is null homotopic. 
\end{proof} 

\begin{example} 
Let $M$ be an $(n-1)$-connected $2n$-dimensional Poincar\'{e} Duality complex 
and let~$G$ be a topological group group. Let 
\(\namedright{P}{}{M}\) 
be a principal $G$-bundle classified by a map of the form 
\(a\colon\nameddright{M}{\pi}{S^{2n}}{\epsilon}{BG}\) 
where $\pi$ is the pinch map to the top cell and $\epsilon$ represents a generator 
of $\pi_{2n-1}(G)$. Depending on $G$, there may be a finite or countably infinite 
number of such principal bundles which are inequivalent. Since $a$ factors through $\pi$, 
the composite $a\circ j$ is null homotopic. Therefore $\Omega a$ is null homotopic 
by Lemma~\ref{PDmap}. Consequently, there is a homotopy equivalence 
$\Omega P\simeq\Omega M\times\Omega G$, 
and this holds independently of the bundle type. 

Carrying on, observe that a map 
\(f\colon\namedright{M}{}{BG}\) 
classifying a principal $G$-bundle has a factorization in the form of $a$ if and only if 
the restriction of $f$ to the $n$-skeleton $\bigvee_{i=1}^{d} S^{n}$ of $M$ is null homotopic. 
This would occur, for example, if $\pi_{n}(BG)\cong\pi_{n-1}(G)$ equals zero. In particular, 
if $M$ is a $2$-connected $6$-manifold and $G$ is a simply-connected simple compact 
Lie group then $\pi_{2}(G)=0$ so every principal $G$-bundle $P$ over $M$ has the 
property that $\Omega P\simeq\Omega M\times\Omega G$. Specializing further, 
if $G=SU(n)$ then the pinch map $\pi$ induces an isomorphism 
$[M,BSU(n)]\cong [S^{6},BSU(n)]$ and if $n\geq 3$ then $\pi_{5}(SU(n))\cong\mathbb{Z}$, 
so there are countably many distinct principal $SU(n)$-bundles over $M$ but all become homotopy 
equivalent after looping. 
\end{example} 

\begin{example} 
This is a variation on the previous example. 
Let $M$ be an $(n-1)$-connected $2n$-dimensional Poincar\'{e} Duality complex and let  
\(a\colon\nameddright{M}{\pi}{S^{2n}}{\epsilon}{M}\) 
be a self-map where $\pi$ is the pinch map to the top cell and $\epsilon$ is any element of 
$\pi_{2n}(M)$. Since $a$ factors through $\pi$, the composite $a\circ j$ is null homotopic. 
Therefore $\Omega a$ is null homotopic by Lemma~\ref{PDmap}. In particular, $a$ induces 
the zero map in homotopy groups. This is despite the fact that $a$ itself need not be 
null homotopic. 
\end{example} 


\section{The collar map for a connected sum} 
\label{sec:collar} 

In this section we give an entirely new example of our methods to demonstrate 
their strength. Let $M$ and $N$ be simply-connected closed $n$-dimensional Poincar\'{e} 
Duality complexes. Let $\overline{M}$ and $\overline{N}$ be the $(n-1)$-skeletons 
of $M$ and $N$ respectively, and let  
\[\nameddright{S^{n-1}}{f_{1}}{\overline{M}}{j_{1}}{M}\] 
\[\nameddright{S^{n-1}}{f_{2}}{\overline{N}}{j_{2}}{N}\] 
be the homotopy cofibrations that attach the top cells to $M$ and $N$. The 
connected sum $M\conn N$ is given by the homotopy cofibration 
\[\llnameddright{S^{n-1}}{f_{1}+f_{2}}{\overline{M}\vee\overline{N}}{}{M\conn N}.\] 
Geometrically, $M\conn N$ is obtained by cutting an $n$-disc out of the interior 
of the top cell in each of $M$ and $N$ and then gluing the two together along the 
boundary of the removed $n$-disc. We can then collapse that collar (the boundary 
of the $n$-disc) to a point to obtain a cofibration 
\[\nameddright{S^{n-1}}{g}{M\conn N}{j}{M\vee N}.\] 

In~\cite[Theorem 5.1]{HaL}, Halperin and Lemaire showed that if $M$ is any simply-connected 
closed, compact Poincar\'{e} Duality complex then the attaching map for the top cell of $M$ 
is rationally inert, implying that the inclusion 
\(\namedright{\overline{M}}{j_{1}}{M}\) 
has the property that, rationally, $\Omega j_{1}$ has a right homotopy inverse. Halperin 
and Lemaire go on~\cite[Theorem 5.4]{HaL} to show that if at least one of 
$\cohlgy{M;\mathbb{Q}}$ or $\cohlgy{N;\mathbb{Q}}$ is not monogenic (generated by a 
single element) then the map 
\(\namedright{M\conn N}{j}{M\vee N}\) 
also has the property that, rationally, $\Omega j$ has a right homotopy inverse. 

We will prove an integral variation of this statement. Integrally, it may not be true 
that $\Omega j_{1}$ and~$\Omega j_{2}$ have right homotopy inverses. However, 
we show that if they do then $\Omega j$ does as well. 

\begin{proposition} 
   \label{collar} 
   Let $M$ and $N$ be simply-connected closed $n$-dimensional Poincar\'{e} Duality 
   complexes. If the inclusions 
   \(\namedright{\overline{M}}{j_{1}}{M}\) 
   and 
   \(\namedright{\overline{N}}{j_{2}}{N}\) 
   have right homotopy inverses after looping, then so does the map 
   \(\namedright{M\conn N}{j}{M\vee N}\). 
   In particular, there is a homotopy equivalence 
   \[\Omega(M\conn N)\simeq\Omega(M\vee N)\times\Omega(\Omega(M\vee N)\ltimes S^{n-1}).\] 
\end{proposition} 

\begin{proof} 
Observe that the composite 
\(\nameddright{\overline{M}\vee\overline{N}}{}{M\conn N}{j}{M\vee N}\) 
is $j_{1}\vee j_{2}$. We will show that $\Omega(j_{1}\vee j_{2})$ has a right 
homotopy inverse, implying that $\Omega j$ also has a right homotopy inverse. 
Granting this, the homotopy cofibration 
\(\nameddright{S^{n-1}}{g}{M\conn N}{j}{M\vee N}\) 
satisfies the hypotheses of Proposition~\ref{specialGTcofib}, and therefore we obtain the 
asserted homotopy equivalence for $\Omega(M\conn N)$.  

It remains to show that $\Omega(j_{1}\vee j_{2})$ has a right homotopy inverse. 
In general, the inclusion of the wedge of simply-connected spaces 
into their product induces a natural homotopy fibration sequence 
\(\nameddright{\Omega X\ast\Omega Y}{}{X\vee Y}{}{X\times Y}\) 
which splits after looping to give a natural homotopy equivalence 
$\Omega(X\vee Y)\simeq\Omega X\times\Omega Y\times\Omega(\Omega X\ast\Omega Y)$. 
In our case we obtain a homotopy fibration diagram 
\[\diagram 
       \Omega\overline{M}\ast\Omega\overline{N}\rto\dto^{\Omega j_{1}\ast\Omega j_{2}} 
             & \overline{M}\vee\overline{N}\rto\dto^{j_{1}\vee j_{2}} 
             & \overline{M}\times\overline{N}\dto^{j_{1}\times j_{2}} \\ 
       \Omega M\ast\Omega N\rto & M\vee N\rto & M\times N 
  \enddiagram\] 
and a homotopy commutative diagram 
\[\diagram 
     \Omega\overline{M}\times\Omega\overline{N}\times\Omega(\Omega\overline{M}\ast\Omega\overline{N}) 
           \rrto\dto^{\Omega j_{1}\times\Omega j_{2}\times\Omega(\Omega j_{1}\ast\Omega j_{2})} 
         & & \Omega(\overline{M}\vee\overline{N})\dto^{\Omega(j_{1}\vee j_{2})} \\ 
     \Omega M\times\Omega N\times\Omega(\Omega M\ast\Omega N)\rrto & & \Omega(M\vee N) 
  \enddiagram\] 
where the horizontal maps are homotopy equivalences. By hypothesis, $\Omega j_{1}$ 
and $\Omega j_{2}$ have right homotopy inverses, and therefore so does 
$\Omega j_{1}\times\Omega j_{2}\times\Omega(\Omega j_{1}\ast\Omega j_{2})$. 
Hence $\Omega(j_{1}\vee j_{2})$ also has a right homotopy inverse. 
\end{proof} 

For example, if $m\geq 2$ then $S^{m}\times S^{2n-m}$ is a Poincar\'{e} Duality 
complex of dimension $2n$, $\overline{S^{m}\times S^{2n-m}}\simeq S^{m}\vee S^{2n-m}$, 
and the inclusion 
\(\namedright{S^{m}\vee S^{2n-m}}{}{S^{m}\times S^{2n-m}}\) 
has a right homotopy inverse after looping. If $m\neq n$ then $S^{m}\times S^{2n-m}$ 
is not an $(n-1)$-connected $2n$-dimensional Poincar\'{e} Duality complex, so it is 
different from the spaces in Example~\ref{decompex1}. Given an $(n-1)$-connected 
$2n$-dimensional Poincar\'{e} Duality complex $M$, 
by Proposition~\ref{collar} there is a homotopy equivalence 
\[\Omega(M\conn(S^{m}\times S^{2n-m}))\simeq\Omega(M\vee(S^{m}\times S^{2n-m}))\times 
       \Omega(\Omega(M\vee(S^{m}\times S^{2n-m}))\ltimes S^{2n-1}).\] 
Similarly, one can consider the connected sum of $S^{m}\times S^{2n+1-m}$ 
and one of the $(n-1)$-connected $(2n+1)$-dimensional manifolds in 
Example~\ref{decompex2}.

\bibliographystyle{amsalpha}

\end{document}